\documentclass[11pt]{article}
\usepackage[ruled,vlined,linesnumbered]{algorithm2e}
\usepackage{amsmath,amsfonts,amssymb,amsthm,latexsym}
\usepackage{epsfig}
\usepackage{color}
\usepackage{enumerate}
\usepackage{hyperref}
\usepackage{enumerate}
\usepackage{graphicx}
\usepackage[cp1250]{inputenc}
\usepackage{tikz}
\usepackage{tkz-graph}
\usepackage{tkz-berge}

\DeclareMathOperator{\TC}{TC}

\usetikzlibrary{arrows.meta,shapes.arrows}
\hypersetup{colorlinks=true, linkcolor=blue, citecolor=blue,
urlcolor=blue}
\usetikzlibrary{calc}

\hypersetup{colorlinks=true}

\hypersetup{colorlinks=true, linkcolor=blue, citecolor=blue,urlcolor=blue}

\hypersetup{colorlinks=true}

\hypersetup{colorlinks=true, linkcolor=blue, citecolor=blue,urlcolor=blue}

\newtheorem{theorem}{Theorem}[section]

\newtheorem{lemma}[theorem]{Lemma}
\newtheorem{observation}[theorem]{Observation}

\theoremstyle{definition}

\theoremstyle{remark}

\begin{document} 

\title{On $k$-coalition in graphs: bounds and exact values}

\author{Bo\v{s}tjan Bre\v{s}ar$\,^{a,b}$, Michael A. Henning$\,^{c}$, \, and \, Babak Samadi$\,^{b}$ \\ \\
$^a$ Faculty of Natural Sciences and Mathematics \\
University of Maribor, Slovenia \\
$^b$ Institute of Mathematics, Physics and Mechanics \\
University of Ljubljana, Slovenia\\
\small \tt Email: bostjan.bresar@um.si \\
\small \tt Email: babak.samadi@imfm.si \\
\\
$^{c}$ Department of Mathematics and Applied Mathematics \\
University of Johannesburg, South Africa\\
\small \tt Email: mahenning@uj.ac.za
}

\date{}
\maketitle

\begin{abstract}
Given a graph $G=\big{(}V(G),E(G)\big{)}$, a set $S\subseteq V(G)$ is called a $k$-dominating set if every vertex in $V(G)\setminus S$ has at least $k$ neighbors in $S$. Two disjoint sets $A,B\subset V(G)$ form a $k$-coalition in $G$ if neither set is a $k$-dominating set in $G$ but their union $A\cup B$ is a $k$-dominating set. A partition $\Omega$ of $V(G)$ is a $k$-coalition partition if each set in $\Omega$ is either a $k$-dominating set of cardinality $k$ or forms a $k$-coalition with another set in $\Omega$. The $k$-coalition number $C_{k}(G)$ equals the maximum cardinality of a $k$-coalition partition of $G$. In this work, we give general upper and lower bounds on this parameter. In particular, we show that if $G$ has minimum degree $\delta \ge 2$ and maximum degree $\Delta \ge 4 \lfloor \delta/2 \rfloor$, then $C_{2}(G) \leq (\Delta-2\lfloor \delta/2 \rfloor+1)(\lfloor \delta/2 \rfloor+1) + \lceil \delta/2 \rceil+1$, and this bound is sharp. If $T$ is a tree of order~$n \ge 2$, then we prove the upper bound $C_{2}(T) \leq \big\lfloor \frac{n}{2}\big\rfloor+1$ and we characterize the extremal trees achieving equality in this bound. We determine the exact value of $C_{k}(G)$ for any cubic graph $G$ and $k\geq2$. Finally, we give the exact value of $C_{k}$ for any complete bipartite graph, which completes a partial result and resolves an issue from an earlier paper.
\end{abstract}



\section{Introduction and preliminaries}

Coalition in graphs has been a hot research topic in recent years. Since 2020, when coalition in graphs was introduced by Haynes, Hedetniemi, Hedentiemi, McRae and Mohan~\cite{hhhmm}, a number of papers were devoted to this graph-theoretic concept. Many variations of the (standard) coalition number of a graph were introduced. One of the most recent variations is the $k$-coalition number of a graph; see~\cite{JAB}.

Throughout this paper, we consider $G$ as a finite simple graph with vertex set $V(G)$ and edge set $E(G)$. We use~\cite{West} as a reference for terminology and notation which are not explicitly defined here. The ({\em open}) {\em neighborhood} of a vertex $v$ is denoted by $N(v)$, and its {\em closed neighborhood} is $N[v]=N(v)\cup \{v\}$. The {\em degree} of a vertex $v$ is $\deg(v)=|N(v)|$. The {\em minimum} and {\em maximum degrees} of $G$ are denoted by $\delta(G)$ and $\Delta(G)$, respectively.

For $k \ge 1$, a set $S\subseteq V(G)$ is a \textit{$k$-dominating set} of $G$ if every vertex $v\in V(G)\setminus S$ has at least $k$ neighbors in $S$, where two vertices are \emph{neighbors} if they are adjacent. A thorough treatise on domination in graphs can be found in~\cite{HaHeHe-20,HaHeHe-21,HaHeHe-23}. For a comprehensive overview of $k$-domination in graphs, the reader can consult the excellent survey by Hansberg and Volkmann~\cite{HaVo-20}. Two disjoint sets $U,V\subseteq V(G)$ form a $k$-\textit{coalition} if neither set is a $k$-dominating set in $G$ but their union is a $k$-dominating set. A $k$-\textit{coalition partition} $\Theta$ of $G$ is a partition of $V(G)$ in which every set is either a $k$-dominating set of cardinality $k$ or forms a $k$-coalition with another set in $\Theta$. The $k$-\textit{coalition number} $C_{k}(G)$ is the maximum cardinality taken over all $k$-coalition partitions of $G$. For the sake of convenience, by a $C_{k}(G)$-partition we mean a $k$-coalition partition of $G$ of maximum cardinality.

In this paper, we give general upper and lower bounds on the $k$-coalition number. In particular, for $k=2$, a sharp upper bound on $C_{2}(G)$ just in terms of minimum and maximum degrees is given, and the bound is widely sharp in the sense that for any even minimum degree there exist families of graphs that attain the bound.  Moreover, if $T$ is a tree of order~$n \ge 2$, we prove the upper bound $C_{2}(T) \leq \big\lfloor \frac{n}{2}\big\rfloor+1$ and we characterize the extremal trees achieving equality in this bound. We determine the exact value of $C_{k}(G)$, where $k\ge 2$ and $G$ is a cubic graph (that is, a graph whose vertices have degree $3$). 

Upper and lower bounds on the $k$-coalition number of complete bipartite graphs were presented  in~\cite{JAB}. In the last section, we first show that the result in~\cite{JAB} is incorrect by presenting an infinite family of counterexamples. Then, we give the exact formula for $C_{k}(K_{s,t})$ that works for all integers $k,s,t$.

We end this section with a remark that if $k\geq2$ and $|V(G)|\ge 2$, then no set in a $C_{k}(G)$-partition is a $k$-dominating set of cardinality $k$. In fact, if there exists such a set $A$ in a $C_{k}(G)$-partition $\Theta$, by replacing $A$ with $A_{1}$ and $A_{2}$ in $\Theta$, where $\{A_{1},A_{2}\}$ is any partition of $A$, we get a $k$-coalition partition of $G$ of cardinality $|C_{k}(G)|+1$, which is impossible.


\section{General bounds}

Clearly, $C_1(K_{n})=n$ for $n\in\mathbb{N}$, and $C_{k}(K_{1})=1$ for $k\in \mathbb{N}$. So, in the following formula, we assume that $k\geq2$ and $n\geq2$:
\begin{equation}\label{complete}
C_{k}(K_{n})=\left \{
\begin{array}{lll}
2, & \textrm{if}\ \ k\geq n;\\
n-k+2, & \textrm{if}\ \ k\leq n-1.
\end{array}
\right.
\end{equation}
Note that the formula $C_{k}(K_{n})=n-k+2$, for $k\in \{2,\ldots,n-1\}$, is given in~\cite{JAB}.

\begin{theorem}\label{lower2}
For any graph $G$ on $n\geq2$ vertices and integer $k\geq2$, $C_{k}(G)\geq \delta(G)-k+3$. Moreover, this bound is sharp.
\end{theorem}
\begin{proof}
The lower bound trivially holds when $\delta(G)<k$, as $C_{k}(G)$ is always greater than or equal to~$2$. So, we may assume that $\delta(G)\geq k$. Let $v\in V(G)$ be a vertex of minimum degree and let $N(v)=\{v_{1},\ldots,v_{\delta(G)}\}$. We set
$S = V(G) \setminus \{v,v_1,\ldots,v_{\delta(G)-k+1}\}$ and
\[
\Omega=\big{\{}\{v\},\{v_{1}\},\ldots,\{v_{\delta(G)-k+1}\},S\big{\}}.
\]

We observe that $\Omega$ is a vertex partition of $G$. Note that $\{v_{i}\}$, for each $i\in[\delta(G)-k+1]$, and $\{v\}$ are not $k$-dominating sets in $G$. Moreover, $v$ has precisely $k-1$ neighbors in $S$. So, $S$ is not a $k$-dominating set in $G$ either. Consider the set $\{v\} \cup S$. Because $|N(v_{i})\cap \{v_{1},\ldots,v_{\delta(G)-k+1}\}|\leq \delta(G)-k$, it follows that $v_{i}$ has at least $k$ neighbors in $\{v\}\cup S$ for each $i\in[\delta(G)-k+1]$. Therefore, $\{v\}$ and $S$ form a $k$-coalition in $G$. Now consider $\{v_{i}\}\cup S$ for any $i\in[\delta(G)-k+1]$. Clearly, $v$ has (precisely) $k$ neighbors in $\{v_{i}\}\cup S$. Now, let $j\in[\delta(G)-k+1]\setminus \{i\}$. Notice that $|N(v_{j})\cap(\{v,v_{1},\ldots,v_{\delta(G)-k+1}\}\setminus \{v_{i}\})|\leq \delta(G)-k$. Therefore, $v_{j}$ has at least $k$ neighbors in $\{v_{i}\}\cup S$. So, $\{v_{i}\}$ is a $k$-coalition partner of $S$ for each $i\in[\delta(G)-k+1]$.

We now infer from the above argument that $\Omega$ is a $k$-coalition partition of $G$. Thus, $C_{k}(G)\geq|\Omega|=\delta(G)-k+3$. Furthermore, the lower bound is sharp for the complete graph $K_{n}$ and integer $k$ with $2\leq k\leq n-1$ due to (\ref{complete}).
\end{proof}

The next result provides us with a natural upper bound on the $k$-coalition number. Despite this fact, the family of graphs for which the upper bound holds with equality is remarkably wide. Recall that the \textit{join} of two graphs $G$ and $H$, written $G\vee H$, is the graph obtained from the disjoint union of $G$ and $H$ by adding the edges $\{gh\mid g\in V(G)\ \mbox{and}\ h\in V(H)\}$.

\begin{theorem}\label{Gen-Char}
For any graph $G$ of order $n\geq3$ and integer $k\in \{3,\ldots,n\}$, $C_{k}(G)\leq n-k+2$. Moreover, equality holds if and only if $n=k$ or $G\cong H\vee K_{n-k+1}$, where $H$ is any graph of order $k-1$.
\end{theorem}
\begin{proof}
Let $\Theta=\{A_{1},\ldots,A_{|\Theta|}\}$ be any $C_{k}(G)$-partition. Without loss of generality, we may assume that $A_{1}$ and $A_{2}$ form a $k$-coalition in $G$. Therefore,
\begin{equation}\label{Upp}
n=|A_{1}\cup A_{2}|+\sum_{i=3}^{|\Theta|}|A_{i}|\geq k+|\Theta|-2.
\end{equation}
This leads to the desired upper bound.

Assume that we have equality in the upper bound. If $|\Theta|=2$, then $n=k$ since $C_{k}(G)=n-k+2$. So, we may assume that $|\Theta|\geq3$. This together with the resulting equality in (\ref{Upp}) imply that $|A_{1}\cup A_{2}|=k$ and that $|A_{i}|=1$ for each $i\in \{3,\ldots,|\Theta|\}$. Let $G_{1}=G[A_{1}\cup A_{2}]$ and $G_{2}=G[\cup_{i=3}^{|\Theta|}A_{i}]$. Since $A_{1}\cup A_{2}$ is a $k$-dominating set of cardinality $k$ in $G$, it follows that every vertex in $G_{2}$ is adjacent to all vertices in $G_{1}$.

On the other hand, since $k\geq3$ and because $|A_{i}|=1$ for each $i\in \{3,\ldots,|\Theta|\}$, it follows that every singleton set $A_{i}$, for $i\in \{3,\ldots,|\Theta|\}$, forms a $k$-coalition with either $A_{1}$ or $A_{2}$. Therefore, one of $A_{1}$ and $A_{2}$, say $A_{1}$, has cardinality $k-1$. Hence, every singleton set $A_{i}$ necessarily forms a $k$-coalition with $A_{1}$. With this in mind, because $A_{1}\cup A_{i}$ is a $k$-dominating set in $G$ of cardinality $k$ for each $i\in[|\Theta|]\setminus \{1\}$, it follows that every vertex in $V(G)\setminus(A_{1}\cup A_{i})$ is adjacent to all vertices in $A_{1}\cup A_{i}$. In particular, it happens that $G_{3}=G[V(G_{2})\cup A_{2}]\cong K_{n-k+1}$. Therefore, $G\cong H\vee K_{n-k+1}$, where $H=G[A_{1}]$.

The converse evidently holds if $n=k$. So, let $n>k$. Hence, $G\cong H\vee K_{n-k+1}$, where $H$ is any graph of order $k-1$. It is then clear from the structure of $G$ that every singleton set $\{x\}$, with $x\in V(K_{n-k+1})$, forms a $k$-coalition with $V(H)$ in $G$. Therefore, $\{V(H)\}\cup \big{\{}\{x\}\mid x\in V(K_{n-k+1})\big{\}}$ is a $k$-coalition partition of $G$ of cardinality $n-k+2$. This leads to the equality $C_{k}(G)=n-k+2$.
\end{proof}

We note that the second statement of Theorem \ref{Gen-Char} fails when $k\in \{1,2\}$. For example, it is obvious that $C(G)\leq n$ for each graph $G$ of order $n$. Moreover, it is easy to see that $C_{2}(C_{4})=4$, while $C_{4}$ does not satisfy the necessity part of the second statement in Theorem \ref{Gen-Char}.

We remark that:\vspace{-2mm}
\begin{itemize}
\item[$\bullet$] a complete characterization of all graphs $G$ with equal order and coalition number $C(G)$ is given in the recent paper \cite{YSZ}, and\vspace{-2mm}
\item[$\bullet$] for a graph $G$ of order $n$, $C_{2}(G)=n$ if and only if for every vertex $v\in V(G)$, there exists a vertex $u\in V(G)$ such that $\{u,v\}$ is a $2$-dominating set in $G$. We observe that the later statement can be verified in polynomial time.
\end{itemize}\vspace{-2mm}

Next, we recall the following statement from~\cite{JAB}. ``\textit{Let $G$ be a graph and let $\Theta$ be a $C_{k}(G)$-partition. If $A\in \Theta$, then $A$ forms a $k$-coalition with at most $\Delta(G)-k+2$ sets in $\Theta$.}" Note that this statement is not correct as it stands when $k\geq \Delta(G)+2$. In what follows, we give a revised version of the above-mentioned statement. It will turn out to be useful in several places in this paper.

\begin{lemma}\label{lemma-k-coal}
Let $G$ be a graph on at least two vertices and let $\Theta$ be a $C_{k}(G)$-partition. If $A\in \Theta$, then $A$ forms a $k$-coalition with at most $\max\{1,\Delta(G)-k+2\}$ sets in $\Theta$.
\end{lemma}
\begin{proof}
If $A\in \Theta$ is a $k$-dominating set in $G$ (it occurs only if $k=|A|=1$), then it forms a $k$-coalition with no set in $\Theta$. So, we may assume that $A$ is not a $k$-dominating set. Therefore, there exists a vertex $v\in V(G)\setminus A$ such that $|N(v)\cap A|<k$. Let $A$ be a $k$-coalition partner of $B\in \Theta$. If $\deg(v)<k$, then $v$ necessarily belongs to $B$ as $A\cup B$ is a $k$-dominating set in $G$. Moreover, $A$ does not form a $k$-coalition with any other set in $\Theta$ as $\deg(v)<k$ and $v\in B$. Therefore, we can assume that $\deg(v)\geq k$. Let $A_{1},\ldots,A_{t}$ be all sets in $\Theta$ that form a $k$-coalition with the set $A$. We have $t\geq1$ as $A$ is not a $k$-dominating set in $G$. Note that if $t=1$, then we are done. So, let $t\geq2$.

Assume first that $v\notin \bigcup_{i=1}^{t}A_{i}$. Since $|N(v)\cap(A\cup A_{i})|\geq k$ for each $i\in[t]$ and because $|N(v)\cap A|<k$, we deduce that
\begin{equation}\label{most1}
\Delta(G)\geq|N(v)|\geq|N(v)\cap(A\cup A_{1})|+\sum_{i=2}^{t}|N(v)\cap A_{i}|\geq k+t-1.
\end{equation}

Now let $v\in \bigcup_{i=1}^{t}A_{i}$. Without loss of generality, we may assume that $v\in A_{t}$ (in such a situation, $N(v)\cap A_{t}$ may be empty). Due to this, an argument analogous to the above leads to
\begin{equation}\label{most2}
\Delta(G)\geq|N(v)|\geq|N(v)\cap(A\cup A_{1})|+\sum_{i=2}^{t}|N(v)\cap A_{i}|\geq k+t-2.
\end{equation}

From (\ref{most1}) and (\ref{most2}), we infer that $A$ is a $k$-coalition partner of at most $\Delta(G)-k+2$ sets in $\Theta$ when $\deg(v)\geq k$. This results in the desired upper bound.
\end{proof}

Before proceeding further, we need the following natural concept. Given a $k$-coalition partition $\Theta= \{V_1, \ldots, V_t\}$ of a graph $G$, the {\em $k$-coalition graph} $C_kG(G,\Theta)$ is the graph whose vertices correspond one-to-one with the sets of $\Theta$, and two vertices $V_i$ and $V_j$ are adjacent in $C_kG(G,\Theta)$ if and only if their corresponding sets $V_i$ and $V_j$ form a $k$-coalition in $G$.

Note that, for various reasons, small values of $k$ (especially $k\in \{1,2\}$) regarding domination-related parameters have attracted more attention from experts in domination theory than large values. One reason is that for large values of $k$, we may have a trivial problem (for example, $C_{k}(G)=2$ for any graph of order $n\geq2$ with $k\geq \Delta(G)+1$). Another reason is that one may obtain stronger results for the small values of $k$ rather than the large ones. In view of this, we turn our attention to the case when $k=2$.

\begin{theorem}\label{thm:general}
For any graph $G$ with $\delta=\delta(G)\geq2$ and $\Delta=\Delta(G)\geq4\big{\lfloor}\frac{\delta}{2}\big{\rfloor}$,
\begin{center}
$C_{2}(G)\leq\Big{(}\Delta-2\left\lfloor\dfrac{\delta}{2}\right\rfloor+1\Big{)}\Big{(}\left\lfloor\dfrac{\delta}{2}\right\rfloor+1\Big{)}+\left\lceil\dfrac{\delta}{2}\right\rceil+1$.
\end{center}
Moreover, the bound is sharp.
\end{theorem}
\begin{proof}
Let $\Omega=\{V_{1},\ldots,V_{|\Omega|}\}$ be a $C_{2}(G)$-partition. Let $u$ be a vertex in $G$ of minimum degree. Without loss of generality, we may assume that $u\in V_{1}$. The following claim will turn out to be useful in the rest of the proof.\vspace{1mm}\\
\textit{Claim A. If $|N_{G}(u)\cap V_{i}|\leq1$ for each $i\in[|\Omega|]$, then $C_{2}(G)\leq \delta+\Delta$.}\vspace{0.75mm}\\
\textit{Proof of Claim A.} Let $\Phi$ be the family of all sets $V_{j}\in \Omega$ such that $N_{G}[u]\cap V_{j}=\emptyset$. We need to differentiate two cases.\vspace{0.5mm}\\
\textit{Case 1.} $|N_{G}(u)\cap V_{1}|=1$. It is then easy to see that each set in $\Phi$ forms a $2$-coalition only with $V_{1}$. On the other hand, $V_{1}$ is a $2$-coalition partner of at most $\Delta$ sets in $\Omega$ by Lemma \ref{lemma-k-coal} with $k=2$. Therefore, $|\Omega|\leq \delta+\Delta$.\vspace{0.5mm}\\
\textit{Case 2.} $N_{G}(u)\cap V_{1}=\emptyset$. Without loss of generality, we may assume that $N_{G}(u)\subseteq \bigcup_{i=2}^{\delta+1}V_{i}$. Note that every set in $\Phi$ necessarily forms a $2$-coalition only with $V_{1}$, and hence $|\Phi|\leq \Delta$ in view of Lemma \ref{lemma-k-coal} with $k=2$. Suppose that $|\Phi|=\Delta$. This implies that $\Phi$ is the family of all $2$-coalition partners of $V_{1}$ due to Lemma~\ref{lemma-k-coal} for $k=2$. We observe that no set in $\{V_{2},\ldots,V_{\delta+1}\}$ forms a $2$-coalition with any set in $\Phi$. Thus we may assume that $V_{2}$ forms a $2$-coalition with $V_{t}$ for some $t\in \{3,\ldots,\delta+1\}$.

Since $V_{1}$ is not a $2$-dominating set in $G$, it follows that there exists a vertex $v\notin V_{1}$ such that $|N_{G}(v)\cap V_{1}|\leq1$. Suppose first that $v\notin \cup_{V_{j}\in \Phi}V_{j}$. This necessarily implies that each $V_{j}\in \Phi$ has a non-empty intersection with $N_{G}(v)$. Then, for any $V_{j}\in \Phi$, we have
\[
|N_{G}(v)|\geq|N_{G}(v)\cap(V_{1}\cup V_{j})|+\sum_{V_{i}\in \Phi\setminus \{V_{j}\}}|N_{G}(v)\cap V_{i}|\geq2+|\Phi|-1=\Delta+1,
\]
a contradiction.  Suppose now that $v\in V_{j}$ for some $V_{j}\in \Phi$. Because $v\notin V_{2}\cup V_{t}$ and since $V_{2}\cup V_{t}$ is a $2$-dominating set in $G$, it follows that $|N_{G}(v)\cap(V_{2}\cup V_{t})|\geq2$. Now, we have
\[
|N_{G}(v)|\geq \sum_{V_{i}\in \Phi\setminus \{V_{j}\}}|N_{G}(v)\cap V_{i}|+|N_{G}(v)\cap(V_{2}\cup V_{t})|\geq|\Phi|-1+2=\Delta+1,
\]
which is impossible. The above argument guarantees that $|\Phi|\leq \Delta-1$. Therefore, $|\Omega|\leq \delta+\Delta$. This completes the proof of Claim A. $(\square)$\vspace{1mm}

\smallskip
With Claim A and the assumption $\Delta\geq4\lfloor \delta/2\rfloor$ in mind, a simple computation implies that
\begin{equation}\label{First}
C_{2}(G)=|\Omega|\leq \delta+\Delta\leq \Big{(}\Delta-2\left\lfloor\dfrac{\delta}{2}\right\rfloor+1\Big{)}\Big{(}\left\lfloor\dfrac{\delta}{2}\right\rfloor+1\Big{)}+\left\lceil\dfrac{\delta}{2}\right\rceil+1.
\end{equation}

From now on, in view of Claim A and the inequality (\ref{First}), we may assume that there exist some sets in $\Omega$ which have at least two vertices from $N_{G}(u)$. With this in mind and renaming the sets if necessary, we let $\Psi=\{V_{1},\ldots,V_{s}\}$, in which $|N_{G}(u)\cap V_{i}|\geq2$ for each $i\in[s]\setminus \{1\}$, and $|N_{G}(u)\cap V_{i}|\leq1$ for each $i>s$. 
It is clear that $s\leq \lfloor \delta/2\rfloor+1$. Let
\[
\Phi=\{V_{j}\in \Omega \mid N_{G}[u]\cap V_{j}=\emptyset\}.
\]

If $\Phi=\emptyset$, then $C_{2}(G)=|\Omega|\leq s+\delta+1-(2s-1)\leq \delta+1$, which is less than or equal to the desired upper bound. So, we may assume that $\Phi\neq \emptyset$. Notice that every set in $\Phi$ necessarily forms a $2$-coalition with $V_{i}$ for some $i\in[s]$. In other words, in the graph $C_{2}G(G,\Omega)$, the vertices $V_{1},\ldots,V_{s}$ dominate all vertices in $\Phi$. In view of this, let $r$ be the smallest number of vertices $V_{j_{1}},\ldots,V_{j_{r}}\in \{V_{1},\ldots,V_{s}\}$ that dominate all sets in $\Phi$. Obviously, $r\in[s]$. For every $i\in[r]$, let $\Omega_{j_{i}}$ be the set of all neighbors of $V_{j_{i}}$ in the graph $C_{2}G(G,\Omega)$ that belong to $\Phi$. By our choice of $r$, we infer that $\Omega_{j_{i}}\nsubseteq \bigcup_{t\in[r]\setminus \{i\}}\Omega_{j_{t}}$ holds for every $i\in [r]$. On the other hand, since $V_{j_{i}}$ is not a $2$-dominating set in $G$, there is a vertex $v\notin V_{j_{i}}$ such that $|N_{G}(v)\cap V_{j_{i}}|\leq1$. Assume that $|N_{G}(v)\cap V_{j_{i}}|=1$. (The possible case when $N_{G}(v)\cap V_{j_{i}}=\emptyset$ uses a similar argument.) Then, at least $|\Omega_{j_{i}}|-1$ sets in $\Omega_{j_{i}}$ have a nonempty intersection with $N_{G}(v)$. (Note that if $v\notin \bigcup_{V_{p}\in \Omega_{j_{i}}}V_{p}$, then each set in $\Omega_{j_{i}}$ intersects $N_{G}(v)$.) On the other hand, there exists a set $V_{k}\in \Omega_{j_{k}}\setminus \bigcup_{t\in[r]\setminus \{k\}}\Omega_{j_{t}}$ for each $k\in[r]\setminus \{i\}$. This implies that $|N_{G}(v)\cap(V_{j_{k}}\cup V_{k})|\geq2$ for at least $r-2$ sets $V_{j_{k}}\in \{V_{j_{1}},\ldots,V_{j_{r}}\}\setminus \{V_{j_{i}}\}$. (Notice that if $v\notin \bigcup_{k\in[r]\setminus \{i\}}(V_{j_{k}}\cup V_{k})$, then $|N_{G}(v)\cap(V_{j_{k}}\cup V_{k})|\geq2$ for each such set $V_{j_{k}}$.)

If $v\in \bigcup_{V_{t}\in \Omega_{j_{i}}}V_{t}$, then $|N_{G}(v)\cap(V_{j_{k}}\cup V_{k})|\geq2$ for each $k\in[r]\setminus \{i\}$. Therefore, $\Delta\geq|N_{G}(v)|\geq1+|\Omega_{j_{i}}|-1+2(r-1)$. If $v\notin \bigcup_{V_{t}\in \Omega_{j_{i}}}V_{t}$, then $\Delta\geq|N_{G}(v)|\geq1+|\Omega_{j_{i}}|+2(r-2)$. In either case, we infer for each $i\in[r]$ that
\begin{equation}\label{Omegi}
|\Omega_{j_{i}}|\leq \Delta-2r+3.
\end{equation}

Since $\{V_{j_{1}},\ldots,V_{j_{r}}\}$ dominates $\Phi$ in the graph $C_{2}G(G,\Omega)$, we deduce that
\begin{equation}\label{Omeg}
C_{2}(G)=|\Omega|\leq s+\sum_{i=1}^{r}|\Omega_{j_{i}}|+\delta+1-(2s-1)=\sum_{i=1}^{r}|\Omega_{j_{i}}|+\delta-s+2.
\end{equation}

Together (\ref{Omegi}) and (\ref{Omeg}) imply that
\begin{equation}\label{Omeg2}
C_{2}(G)\leq r(\Delta-2r+3)+\delta-s+2\leq r(\Delta-2r+3)+\delta-r+2=f(r).
\end{equation}

If $r=\lfloor \delta/2\rfloor+1$, then we get the desired upper bound. On the other hand, since $\Delta\geq4\lfloor \delta/2\rfloor$, it follows that $f$ is an increasing function on $[1,\lfloor \delta/2\rfloor]$. Thus, if $r\leq \lfloor \delta/2\rfloor$, then
\begin{align*}
C_{2}(G) & \leq f(r)\leq f\Big{(}\left\lfloor \frac{\delta}{2}\right\rfloor\Big{)}=f\Big{(}\left\lfloor \frac{\delta}{2}\right\rfloor+1\Big{)}+4\left\lfloor \frac{\delta}{2}\right\rfloor-\Delta\\
&\leq f\Big{(}\left\lfloor \frac{\delta}{2}\right\rfloor+1\Big{)}=\Big{(}\Delta-2\left\lfloor\frac{\delta}{2}\right\rfloor+1\Big{)}\Big{(}\left\lfloor\frac{\delta}{2}\right\rfloor+1\Big{)}+\left\lceil\frac{\delta}{2}\right\rceil+1,
\end{align*}
as desired.

To see that the upper bound is sharp, we introduce the graphs $G(d)$, where $d\ge 4$, as follows. Take the disjoint union of $3d-1$ copies of the cycle $C_4$ and two copies of the star $K_{1,d}$. Let $$A_1,\ldots,A_d,B_2,\ldots,B_d,C_1,\ldots,C_d$$
be the copies of $C_4$, where the vertices of $A_i$ form the bipartition $\big{\{}\{a_1^i,a_2^i\},\{a_3^i,a_4^i\}\big{\}}$, the vertices of $B_i$ form the bipartition $\big{\{}\{b_1^i,b_2^i\},\{b_3^i,b_4^i\}\big{\}}$, and the vertices of $C_i$ form the bipartition $\big{\{}\{c_1^i,c_2^i\},\{c_3^i,c_4^i\}\big{\}}$.
The two stars $K_{1,d}$ have vertices $x$ and $y$ as their centers, and $X=\{x_1,\ldots,x_d\}$ and $Y=\{y_1,\ldots,y_d\}$ are their sets of leaves, respectively. The graph $G(d)$ is obtained from the described disjoint union by connecting $x_i$ with $a_4^i$ for each $i\in[d]$, connecting $x_i$ with $b_1^i$ and $b_2^i$ for each $i\in[d]\setminus\{1\}$, connecting $y_i$ with $b_3^i$ and $b_4^i$ for each $i\in[d]\setminus\{1\}$, and connecting $y_i$ with $c_1^i$ for each $i\in[d]$. (See Figure \ref{Fig1}.)
Note that all vertices in $\bigcup_{i=1}^d{\{a_1^i,a_2^i,a_3^i,c_2^i,c_3^i,c_4^i\}}\bigcup\{x_1,y_1\}$ have degree $2$ in $G(d)$. The vertices in $\bigcup_{i=1}^d{\{a_4^i,c_1^i\}}\bigcup(\bigcup_{i=2}^d{\{b_1^i,b_2^i,b_3^i,b_4^i\}})$ have degree $3$ in $G(d)$, while the vertices in $(X\cup Y)\setminus\{x_1,y_1\}$ have degree $4$ in $G(d)$.
Finally, $x$ and $y$ have degree $d$ in $G(d)$. Thus, $\delta=\delta\big{(}G(d)\big{)}=2$ and $\Delta=\Delta\big{(}G(d)\big{)}=d$. Since $\Delta=d\ge 4=2\delta$, the condition of the theorem is satisfied.

Now, let us present a partition $\Omega=\{V_1,\ldots,V_{|\Omega|}\}$ of $V\big{(}G(d)\big{)}$, where $|\Omega|=2d$. Let
\[
V_1=\{a_1^1,\ldots,a_1^d,a_2^1,\ldots,a_2^d,b_1^2,\ldots,b_1^d,b_2^2,\ldots,b_2^d, c_1^1,\ldots,c_1^d,c_2^1,\ldots,c_2^d,x_1,y\},
\]
\[
V_{d+1}=\{a_3^1,\ldots,a_3^d,a_4^1,\ldots,a_4^d,b_3^2,\ldots,b_3^d,b_4^2,\ldots,b_4^d, c_3^1,\ldots,c_3^d,c_4^1,\ldots,c_4^d,y_1,x\},
\]
and $V_i=\{x_i\}$, $V_{d+i}=\{y_i\}$ for all $i\in\{2,\ldots,d\}$. Note that $V_1$ is not a 2-dominating set in $G(d)$, because $x\in V_{d+1}$ has only one neighbor in $V_1$, namely $x_1$. Similarly, $V_{d+1}$ is not a 2-dominating set in $G(d)$, since $y\in V_{1}$ has only one neighbor in $V_{d+1}$, namely $y_1$. For the remaining sets $V_i$ it is easy to verify that they are not $2$-dominating sets in $G(d)$ (say, by considering the neighborhood of $a_1^1$). Now, note that $V_1$ forms a 2-coalition with each of the sets $V_2,\ldots,V_d$, while $V_{d+1}$ forms a 2-coalition with each of the sets $V_{d+2},\ldots,V_{2d}$. Therefore, $G(d)$ has a 2-coalition partition of cardinality $2d$. Since
\[
\Big{(}\Delta-2\left\lfloor\dfrac{\delta}{2}\right\rfloor+1\Big{)}\Big{(}\left\lfloor\dfrac{\delta}{2}\right\rfloor+1\Big{)} + \left\lceil\dfrac{\delta}{2}\right\rceil+1 
= (d-2+1)(1+1)+1+1=2d
\]
we infer that $G(d)$ attains the upper bound of the theorem. This completes the proof.
\end{proof}

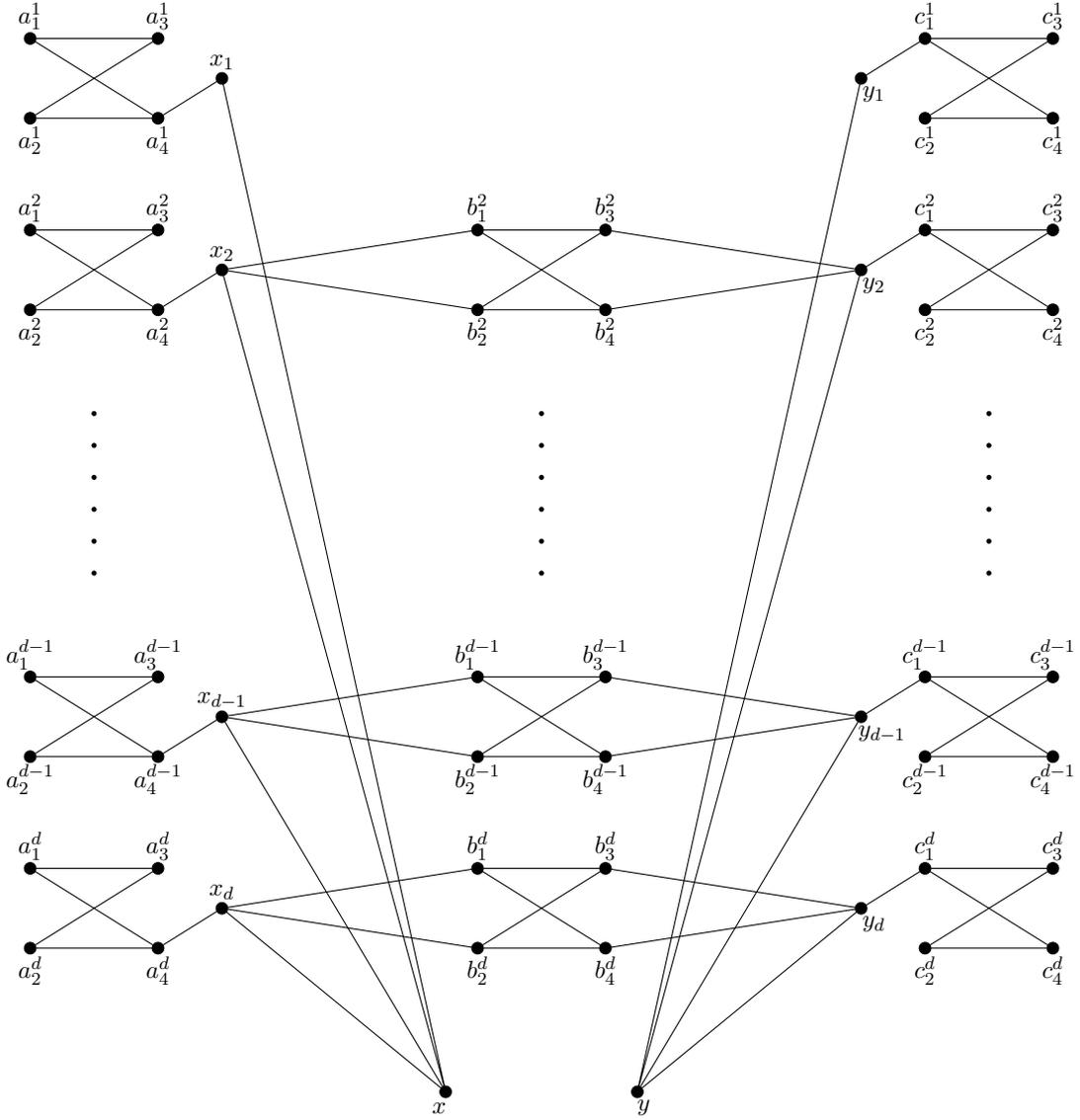
\begin{figure}[ht!]
\centering
\begin{tikzpicture}[scale=0.43, transform shape]
\node [draw, shape=circle, fill=black] (a_{1}^{1}) at (-16,0) {};
\node [draw, shape=circle, fill=black] (a_{3}^{1}) at (-12,0) {};
\node [draw, shape=circle, fill=black] (a_{2}^{1}) at (-16,-2.5) {};
\node [draw, shape=circle, fill=black] (a_{4}^{1}) at (-12,-2.5) {};
\node [draw, shape=circle, fill=black] (x_{1}) at (-10,-1.25) {};
\draw (a_{1}^{1})--(a_{3}^{1})--(a_{2}^{1})--(a_{4}^{1})--(a_{1}^{1});
\draw (x_{1})--(a_{4}^{1});

\node [scale=2] at (-16,0.7) {$a_{1}^{1}$};
\node [scale=2] at (-12,0.7) {$a_{3}^{1}$};
\node [scale=2] at (-16,-3.2) {$a_{2}^{1}$};
\node [scale=2] at (-12,-3.2) {$a_{4}^{1}$};
\node [scale=2] at (-10,-0.75) {$x_{1}$};
\node [scale=2] at (-10,-6.75) {$x_{2}$};
\node [scale=2] at (-10,-20.75) {$x_{d-1}$};
\node [scale=2] at (-10,-26.75) {$x_{d}$};

\node [draw, shape=circle, fill=black] (a_{1}^{2}) at (-16,-6) {};
\node [draw, shape=circle, fill=black] (a_{3}^{2}) at (-12,-6) {};
\node [draw, shape=circle, fill=black] (a_{2}^{2}) at (-16,-8.5) {};
\node [draw, shape=circle, fill=black] (a_{4}^{2}) at (-12,-8.5) {};
\node [draw, shape=circle, fill=black] (x_{2}) at (-10,-7.25) {};
\draw (a_{1}^{2})--(a_{3}^{2})--(a_{2}^{2})--(a_{4}^{2})--(a_{1}^{2});
\draw (x_{2})--(a_{4}^{2});
\node [scale=2] at (-16,-5.3) {$a_{1}^{2}$};
\node [scale=2] at (-12,-5.3) {$a_{3}^{2}$};
\node [scale=2] at (-16,-9.2) {$a_{2}^{2}$};
\node [scale=2] at (-12,-9.2) {$a_{4}^{2}$};

\node [scale=3.25] at (-14,-11.75) {\Large $.$};
\node [scale=3.25] at (-14,-12.75) {\Large $.$};
\node [scale=3.25] at (-14,-13.75) {\Large $.$};
\node [scale=3.25] at (-14,-14.75) {\Large $.$};
\node [scale=3.25] at (-14,-15.75) {\Large $.$};
\node [scale=3.25] at (-14,-16.75) {\Large $.$};

\node [draw, shape=circle, fill=black] (a_{1}^{d-1}) at (-16,-20) {};
\node [draw, shape=circle, fill=black] (a_{3}^{d-1}) at (-12,-20) {};
\node [draw, shape=circle, fill=black] (a_{2}^{d-1}) at (-16,-22.5) {};
\node [draw, shape=circle, fill=black] (a_{4}^{d-1}) at (-12,-22.5) {};
\node [draw, shape=circle, fill=black] (x_{d-1}) at (-10,-21.25) {};
\draw (a_{1}^{d-1})--(a_{3}^{d-1})--(a_{2}^{d-1})--(a_{4}^{d-1})--(a_{1}^{d-1});
\draw (x_{d-1})--(a_{4}^{d-1});
\node [scale=2] at (-16,-19.3) {$a_{1}^{d-1}$};
\node [scale=2] at (-12,-19.3) {$a_{3}^{d-1}$};
\node [scale=2] at (-16,-23.2) {$a_{2}^{d-1}$};
\node [scale=2] at (-12,-23.2) {$a_{4}^{d-1}$};

\node [draw, shape=circle, fill=black] (a_{1}^{d}) at (-16,-26) {};
\node [draw, shape=circle, fill=black] (a_{3}^{d}) at (-12,-26) {};
\node [draw, shape=circle, fill=black] (a_{2}^{d}) at (-16,-28.5) {};
\node [draw, shape=circle, fill=black] (a_{4}^{d}) at (-12,-28.5) {};
\node [draw, shape=circle, fill=black] (x_{d}) at (-10,-27.25) {};
\draw (a_{1}^{d})--(a_{3}^{d})--(a_{2}^{d})--(a_{4}^{d})--(a_{1}^{d});
\draw (x_{d})--(a_{4}^{d});
\node [scale=2] at (-16,-25.3) {$a_{1}^{d}$};
\node [scale=2] at (-12,-25.3) {$a_{3}^{d}$};
\node [scale=2] at (-16,-29.2) {$a_{2}^{d}$};
\node [scale=2] at (-12,-29.2) {$a_{4}^{d}$};


\node [draw, shape=circle, fill=black] (b_{1}^{2}) at (-2,-6) {};
\node [draw, shape=circle, fill=black] (b_{3}^{2}) at (2,-6) {};
\node [draw, shape=circle, fill=black] (b_{2}^{2}) at (-2,-8.5) {};
\node [draw, shape=circle, fill=black] (b_{4}^{2}) at (2,-8.5) {};
\draw (b_{1}^{2})--(b_{3}^{2})--(b_{2}^{2})--(b_{4}^{2})--(b_{1}^{2});
\draw (b_{1}^{2})--(x_{2})--(b_{2}^{2});
\node [scale=2] at (-2,-5.3) {$b_{1}^{2}$};
\node [scale=2] at (2,-5.3) {$b_{3}^{2}$};
\node [scale=2] at (-2,-9.2) {$b_{2}^{2}$};
\node [scale=2] at (2,-9.2) {$b_{4}^{2}$};

\node [scale=3.25] at (0,-11.75) {\Large $.$};
\node [scale=3.25] at (0,-12.75) {\Large $.$};
\node [scale=3.25] at (0,-13.75) {\Large $.$};
\node [scale=3.25] at (0,-14.75) {\Large $.$};
\node [scale=3.25] at (0,-15.75) {\Large $.$};
\node [scale=3.25] at (0,-16.75) {\Large $.$};

\node [draw, shape=circle, fill=black] (b_{1}^{d-1}) at (-2,-20) {};
\node [draw, shape=circle, fill=black] (b_{3}^{d-1}) at (2,-20) {};
\node [draw, shape=circle, fill=black] (b_{2}^{d-1}) at (-2,-22.5) {};
\node [draw, shape=circle, fill=black] (b_{4}^{d-1}) at (2,-22.5) {};
\draw (b_{1}^{d-1})--(b_{3}^{d-1})--(b_{2}^{d-1})--(b_{4}^{d-1})--(b_{1}^{d-1});
\draw (b_{1}^{d-1})--(x_{d-1})--(b_{2}^{d-1});
\node [scale=2] at (-2,-19.3) {$b_{1}^{d-1}$};
\node [scale=2] at (2,-19.3) {$b_{3}^{d-1}$};
\node [scale=2] at (-2,-23.2) {$b_{2}^{d-1}$};
\node [scale=2] at (2,-23.2) {$b_{4}^{d-1}$};

\node [draw, shape=circle, fill=black] (b_{1}^{d}) at (-2,-26) {};
\node [draw, shape=circle, fill=black] (b_{3}^{d}) at (2,-26) {};
\node [draw, shape=circle, fill=black] (b_{2}^{d}) at (-2,-28.5) {};
\node [draw, shape=circle, fill=black] (b_{4}^{d}) at (2,-28.5) {};
\draw (b_{1}^{d})--(b_{3}^{d})--(b_{2}^{d})--(b_{4}^{d})--(b_{1}^{d});
\draw (b_{1}^{d})--(x_{d})--(b_{2}^{d});
\node [scale=2] at (-2,-25.3) {$b_{1}^{d}$};
\node [scale=2] at (2,-25.3) {$b_{3}^{d}$};
\node [scale=2] at (-2,-29.2) {$b_{2}^{d}$};
\node [scale=2] at (2,-29.2) {$b_{4}^{d}$};


\node [draw, shape=circle, fill=black] (c_{1}^{1}) at (12,0) {};
\node [draw, shape=circle, fill=black] (c_{3}^{1}) at (16,0) {};
\node [draw, shape=circle, fill=black] (c_{2}^{1}) at (12,-2.5) {};
\node [draw, shape=circle, fill=black] (c_{4}^{1}) at (16,-2.5) {};
\node [draw, shape=circle, fill=black] (y_{1}) at (10,-1.25) {};
\draw (c_{1}^{1})--(c_{3}^{1})--(c_{2}^{1})--(c_{4}^{1})--(c_{1}^{1});
\draw (y_{1})--(c_{1}^{1});
\node [scale=2] at (12,0.7) {$c_{1}^{1}$};
\node [scale=2] at (16,0.7) {$c_{3}^{1}$};
\node [scale=2] at (12,-3.2) {$c_{2}^{1}$};
\node [scale=2] at (16,-3.2) {$c_{4}^{1}$};
\node [scale=2] at (10.4,-1.75) {$y_{1}$};
\node [scale=2] at (10.4,-7.75) {$y_{2}$};
\node [scale=2] at (10.65,-21.79) {$y_{d-1}$};
\node [scale=2] at (10.4,-27.75) {$y_{d}$};

\node [draw, shape=circle, fill=black] (c_{1}^{2}) at (12,-6) {};
\node [draw, shape=circle, fill=black] (c_{3}^{2}) at (16,-6) {};
\node [draw, shape=circle, fill=black] (c_{2}^{2}) at (12,-8.5) {};
\node [draw, shape=circle, fill=black] (c_{4}^{2}) at (16,-8.5) {};
\node [draw, shape=circle, fill=black] (y_{2}) at (10,-7.25) {};
\draw (c_{1}^{2})--(c_{3}^{2})--(c_{2}^{2})--(c_{4}^{2})--(c_{1}^{2});
\draw (y_{2})--(c_{1}^{2});
\draw (b_{3}^{2})--(y_{2})--(b_{4}^{2});
\node [scale=2] at (12,-5.3) {$c_{1}^{2}$};
\node [scale=2] at (16,-5.3) {$c_{3}^{2}$};
\node [scale=2] at (12,-9.2) {$c_{2}^{2}$};
\node [scale=2] at (16,-9.2) {$c_{4}^{2}$};

\node [scale=3.25] at (14,-11.75) {\Large $.$};
\node [scale=3.25] at (14,-12.75) {\Large $.$};
\node [scale=3.25] at (14,-13.75) {\Large $.$};
\node [scale=3.25] at (14,-14.75) {\Large $.$};
\node [scale=3.25] at (14,-15.75) {\Large $.$};
\node [scale=3.25] at (14,-16.75) {\Large $.$};

\node [draw, shape=circle, fill=black] (c_{1}^{d-1}) at (12,-20) {};
\node [draw, shape=circle, fill=black] (c_{3}^{d-1}) at (16,-20) {};
\node [draw, shape=circle, fill=black] (c_{2}^{d-1}) at (12,-22.5) {};
\node [draw, shape=circle, fill=black] (c_{4}^{d-1}) at (16,-22.5) {};
\node [draw, shape=circle, fill=black] (y_{d-1}) at (10,-21.25) {};
\draw (c_{1}^{d-1})--(c_{3}^{d-1})--(c_{2}^{d-1})--(c_{4}^{d-1})--(c_{1}^{d-1});
\draw (y_{d-1})--(c_{1}^{d-1});
\draw (b_{3}^{d-1})--(y_{d-1})--(b_{4}^{d-1});
\node [scale=2] at (12,-19.3) {$c_{1}^{d-1}$};
\node [scale=2] at (16,-19.3) {$c_{3}^{d-1}$};
\node [scale=2] at (12,-23.2) {$c_{2}^{d-1}$};
\node [scale=2] at (16,-23.2) {$c_{4}^{d-1}$};

\node [draw, shape=circle, fill=black] (c_{1}^{d}) at (12,-26) {};
\node [draw, shape=circle, fill=black] (c_{3}^{d}) at (16,-26) {};
\node [draw, shape=circle, fill=black] (c_{2}^{d}) at (12,-28.5) {};
\node [draw, shape=circle, fill=black] (c_{4}^{d}) at (16,-28.5) {};
\node [draw, shape=circle, fill=black] (y_{d}) at (10,-27.25) {};
\draw (c_{1}^{d})--(c_{3}^{d})--(c_{2}^{d})--(c_{4}^{d})--(c_{1}^{d});
\draw (y_{d})--(c_{1}^{d});
\draw (b_{3}^{d})--(y_{d})--(b_{4}^{d});
\node [scale=2] at (12,-25.3) {$c_{1}^{d}$};
\node [scale=2] at (16,-25.3) {$c_{3}^{d}$};
\node [scale=2] at (12,-29.2) {$c_{2}^{d}$};
\node [scale=2] at (16,-29.2) {$c_{4}^{d}$};


\node [draw, shape=circle, fill=black] (x) at (-3,-33) {};
\draw (x_{d})--(x)--(x_{d-1});
\draw (x_{2})--(x)--(x_{1});
\node [scale=2] at (-3.2,-33.5) {$x$};


\node [draw, shape=circle, fill=black] (y) at (3,-33) {};
\draw (y_{d})--(y)--(y_{d-1});
\draw (y_{2})--(y)--(y_{1});
\node [scale=2] at (3.2,-33.5) {$y$};

\end{tikzpicture}
\caption{{\small The graph $G(d)$ with $C_{2}\big{(}G(d)\big{)}=2d$.}}
\label{Fig1}
\end{figure}

Using a similar idea as in the above proof, for each even integer $\delta\geq2$, one can construct an infinite family of graphs, with minimum degree equal to $\delta$, attaining the upper bound of Theorem~\ref{thm:general}. To do so, let $H(r)$ denote the {\em cocktail-party graph}, which is the graph obtained from the complete graph $K_{2r}$ by removing edges of a perfect matching. Note that $\omega\big{(}H(r)\big{)}=r$, in which $\omega$ stands for the clique number, and that there exist two (vertex and edge) disjoint complete subgraphs of $H(r)$ of order $r$. If $Q$ is a graph isomorphic to $H(r)$, then by $Q'$ and $Q''$ we will denote the two disjoint complete subgraphs of $Q$, each of order $r$.

Let $\delta\ge 4$ be an even integer, $r=\delta/2+1$, and let $\Delta\ge 2\delta$. Take the disjoint union of $r$ stars $K_{1,\Delta}$, with centers $w_1,\ldots,w_r$, and let $W_1,\ldots,W_r$ be the sets of their leaves, respectively. (Note that there are $\Delta\cdot r$ vertices in $W_1\cup\cdots\cup W_r$.) For each vertex $q$ in $W_1\cup\cdots\cup W_r$, take two disjoint copies of the graph $Q=H(r)$, denoted by $Q_1$ and $Q_2$, and connect $q$ to the vertices of $V(Q_1')\cup V(Q_2')$. In this way, the degree of $q$ is $2r+1$, and the degrees of vertices in $V(Q_1)\cup V(Q_2)$ are in $\{2r-2,2r-1\}$. On the other hand, the degree of vertex $w_i$, where $i\in [r]$, is $\Delta$. Make the graph connected by adding $r-1$ edges connecting vertices of the copies of $H(r)$ that correspond to the leaves in the sets $W_1,\ldots, W_r$. Let the resulting graph be denoted by $G(\delta,\Delta)$. Note that $\delta\big{(}G(\delta,\Delta)\big{)}=\delta=2r-2$ and $\Delta\big{(}G(\delta,\Delta)\big{)}=\Delta$.

Now, let us present a partition $\Omega=\{V_1,\ldots,V_{|\Omega|}\}$ of $V\big{(}G(\delta,\Delta)\big{)}$ that we claim it is a $2$-coalition partition. For each copy $Q$ of the graph $H(r)$, let the vertices in $V(Q')\cup V(Q'')$ be distributed to the sets $V_1,\ldots, V_r$, in such a way that the endvertices of the $i$th edge of the removed perfect matching belong to $V_{i}$ for each $i\in[r]$. Let $w_i\in V_i$ for $i\in [r]$. Now, vertices in the sets $W_i$ are distributed as follows. Let one vertex of $W_i$ belong to $V_{i+1}$ and let two vertices of $W_i$ belong to $V_j$ for all $j\in [r]\setminus\{i,i+1\}$, where we identify $V_{r+1}$ with $V_{1}$. The remaining $\Delta-2r+3$ vertices of $W_i$ are distributed to the sets $V_{r+(i-1)(\Delta-2r+3)+j}$, for all $j\in[\Delta-2r+3]$ and $i\in[r]$. One can verify that no set $V_i$ is a $2$-dominating set. Moreover, for all $i\in[r]$, the set $V_{i+1}$ forms a $2$-coalition with the sets $V_{r+(i-1)(\Delta-2r+3)+1},V_{r+(i-1)(\Delta-2r+3)+2},\ldots, V_{r+i(\Delta-2r+3)}$, where we identify $V_{r+1}$ with $V_{1}$. Hence, $\Omega$ is a $2$-coalition partition of the graph $G(\delta,\Delta)$, and its cardinality is 
\[
|\Omega|=r(\Delta-2r+3)+r= \left( \frac{\delta}{2}+1 \right)\left( \Delta-2\cdot\frac{\delta}{2}+1 \right) + \frac{\delta}{2}+1,
\]
which coincides with the upper bound in Theorem~\ref{thm:general}. The following observation readily follows.

\begin{observation}
For every even positive integer $\delta$ and every integer $\Delta\ge 2\delta$, there exist graphs $G$, with $\delta(G)=\delta$ and $\Delta(G)=\Delta$, such that 
\[
C_{2}(G) = \Big{(}\Delta-2\left\lfloor\dfrac{\delta}{2}\right\rfloor+1\Big{)}\Big{(}\left\lfloor\dfrac{\delta}{2}\right\rfloor+1\Big{)} + \left\lceil\dfrac{\delta}{2}\right\rceil+1.
\]
\end{observation}


\section{Trees}

The authors of \cite{bhp} characterized the trees $T$ of order $n$ for which equality holds in the trivial upper bound $C(T)\leq n$ as follows.

\begin{theorem}\emph{(\cite{bhp})}
If $T$ is a tree of order $n$, then $C(T)=n$ if and only if $T$ is a path of order at most~$4$.
\end{theorem}

Here, we give a nontrivial upper bound on the $2$-coalition number of a tree, just in terms of its order, and characterize the extremal trees for the upper bound. In \cite{JAB}, the $2$-coalition number of a tree $T$ of order $n$ was bounded from above by $n/2+1$ (that is, $C_{2}(T)\leq \lfloor n/2\rfloor+1$).

Recall that for $a,b\geq1$, a \textit{double star} $S_{a,b}$ is a tree with exactly two (adjacent) vertices that are not leaves, one of which has $a$ leaf neighbors and the other $b$ leaf neighbors.

\begin{theorem}\label{upper-tree}
If $T$ is a tree of order $n\geq2$, then $C_{2}(T)\leq \big\lfloor \frac{n}{2}\big\rfloor+1$, with equality if and only if $T\in \{P_{2},P_{3},P_{4},P_{5},S_{1,2}\}$.
\end{theorem}
\begin{proof}
In order to characterize all trees for which the upper bound holds with equality, we first give a proof for the inequality $C_{2}(T)\leq \lfloor n/2\rfloor+1$ for any tree $T$ of order $n\geq2$. Let $\Omega$ and $L(T)$ be a $C_{2}(T)$-partition and the set of leaves in $T$, respectively, and let $\ell(T)=|L(T)|$.

Suppose that $L(T)\cap V_{i}\neq \emptyset$ for three sets $V_{1},V_{2},V_{3}\in \Omega$. We may assume, without loss of generality, that $V_{1}$ forms a $2$-coalition with a set $U\in \Omega\setminus \{V_{3}\}$. In such a situation, any leaf $x\in V_{3}$ has at least two neighbors in $U\cup V_{1}$, a contradiction. Thus, $L(T)$ intersects at most two sets in~$\Omega$. Assume that $L(T)\cap V_{i}\neq \emptyset$ for two sets $V_{1},V_{2}\in \Omega$. If there exists a set $U\in \Omega\setminus \{V_{1},V_{2}\}$, then we may assume that it forms a $2$-coalition with a set $W\in \Omega\setminus \{V_{1}\}$. This is in contradiction with the existence of at least one leaf in $V_{1}$ as $U\cup W$ is a $2$-dominating set in $T$. Therefore, $C_{2}(T)=2\leq \lfloor n/2\rfloor+1$. So, from now on, we can assume that $L(T)\subseteq V$ for a unique set $V\in \Omega$. This necessarily implies that each set in $\Omega\setminus \{V\}$ forms a $2$-coalition with $V$. Notice that the set $V$ is a $2$-coalition partner of at most $\Delta(T)$ sets in $\Omega$ due to Lemma \ref{lemma-k-coal} for $k=2$. Therefore, $C_{2}(T)\leq1+\Delta(T)$. On the other hand, it is readily observed that $C_{2}(T)\leq1+n-|V|\leq1+n-\ell(T)$. By taking the last two upper bounds on $C_{2}(T)$ into account, we get
\begin{equation}\label{tree}
C_{2}(T)\leq \frac{n-\ell(T)+\Delta(T)}{2}+1\leq \frac{n}{2}+1,
\end{equation}
leading to the desired upper bound.

It is easy to check that $C_{2}(T)=\lfloor n/2\rfloor+1$ if $T$ is any tree in the family $\{P_{2},P_{3},P_{4},P_{5},S_{1,2}\}$. Conversely, let equality hold in the upper bound for $T\ncong P_{2}$. We root $T$ at a vertex $r$ of maximum degree and let $L(T)=\{u_{1},\ldots,u_{\ell(T)}\}$.

Assume first that $n$ is even. If $L(T)$ intersects two sets in $\Omega$, then $C_{2}(T)=2$. Hence, $n=2$, which contradicts our assumption $T\ncong P_{2}$. Therefore, $L(T)\subseteq V$ for a unique set $V\in \Omega$. On the other hand, we have $\Delta(T)=\ell(T)$ because the inequality chain (\ref{tree}) necessarily holds with equality. It follows that $T$ is a union of $\ell(T)$ paths $Q_{1},\ldots,Q_{\ell(T)}$ having $r$ as a common endvertex, in which $u_{1},\ldots,u_{\ell(T)}$ are the other endvertices, respectively. Because (\ref{tree}) holds with equality, the resulting equality $C_{2}(T)=1+n-|V|=1+n-\ell(T)$ guarantees that $L(T)=V$ and that $\Omega=\big{\{}L(T)\big{\}}\bigcup \big{\{}\{v\}\mid v\notin L(T)\big{\}}$. In fact, every singleton set $\{v\}\in \Omega$ must form a $2$-coalition only with $L(T)$. If $|V(Q_{1})|=\cdots=|V(Q_{\ell(T)})|=2$, then $T\cong K_{1,\ell(T)}$, and hence $C_{2}(T)=2<n/2+1$, a contradiction. Therefore, at least one of $Q_{1},\ldots,Q_{\ell(T)}$, say $Q_{1}$, has at least three vertices. Suppose that $|V(Q_{1})|\geq4$ and that $u,v\in V(Q_{1})\setminus \{r,u_{1}\}$, in which $v$ is the support vertex adjacent to $u_{1}$. Then, $u$ has at most one neighbor in $L(T)\cup \{v\}$, contradicting the fact that $\{v\}$ and $L(T)$ form a $2$-coalition in $T$. Therefore, $|V(Q_{1})|=3$. Because $r$ has at least two neighbors in $L(T)\cup \{v\}$, it follows that $r$ is adjacent to at least one leaf. If there is a path $Q_{i}$ with $|V(Q_{i})|\geq3$ for some $i\in[\ell(T)]\setminus \{1\}$, then any internal vertex of $Q_{i}$ is not $2$-dominated by $L(T)\cup \{v\}$, a contradiction. Hence, all other neighbors of $r$ are leaves. The above argument shows that $T\cong S_{1,\ell(T)-1}$. So, by taking into account the fact that $L(T)\in \Omega$, we have $C_{2}(S_{1,\ell(T)-1})=3=(\ell(T)+2)/2+1$, and hence $\ell(T)=2$. Therefore, $T\cong P_{4}$.

Now let $n$ be odd. If $L(T)$ intersects two sets in $\Omega$, we have $C_{2}(T)=2=(n-1)/2+1$, and hence $T\cong P_{3}$. So, let $L(T)\subseteq V$ for a unique set $V\in \Omega$. Suppose, to the contrary, that $\ell(T)\geq \Delta(T)+2$. Then, we infer from (\ref{tree}) that $C_{2}(T)\leq n/2<\lfloor n/2\rfloor+1$, which is impossible. Therefore, $\ell(T)\in \{\Delta(T),\Delta(T)+1\}$. Due to this, we consider two cases.\vspace{0.75mm}\\
\textit{Case 1.} $\ell(T)=\Delta(T)+1$. In such a situation, $T$ is obtained from $\ell(T)$ paths $R_{1},\ldots,R_{\ell(T)}$ having $r$ as a common endvertex, in which $u_{1},\ldots,u_{\ell(T)}$ are the other endvertices, respectively. Moreover, precisely two paths, say $R_{1}$ and $R_{2}$, have a common internal vertex. If $C_{2}(T)\leq n-\ell(T)$, then it results in the impossible inequality $C_{2}(T)\leq n/2$. Therefore, $C_{2}(T)=1+n-\ell(T)$, which in turn implies that $V=L(T)$ and that $\Omega=\big{\{}L(T)\big{\}}\bigcup \big{\{}\{v\}\mid v\notin L(T)\big{\}}$. Moreover, $L(T)$ is the only $2$-coalition partner of $\{v\}$ for each $v\in V(T)\setminus L(T)$. An argument similar to that of the case $n\equiv0$ (mod $2$) shows that $|V(R_{i})|\leq3$ for each $i\in[\ell(T)]$. In particular, $|V(R_{1})|=|V(R_{2})|=3$. Let $v$ be the unique common internal vertex of $R_{1}$ and $R_{2}$. If there exists an index $i\in[\ell(T)]\setminus \{1,2\}$ such that $|V(R_{i})|=3$, then the internal vertex of $R_{i}$ is not $2$-dominated by $L(T)\cup \{v\}$. Therefore, $|V(R_{i})|=2$ for each $i\in[\ell(T)]\setminus \{1,2\}$. This shows that $T$ is isomorphic to the double star $S_{2,\ell(T)-2}$. With this in mind, $C_{2}(T)=3=(\ell(T)+1)/2+1$ implies that $\ell(T)=3$, and hence $T\cong S_{1,2}$.\vspace{0.75mm}\\
\textit{Case 2.} $\ell(T)=\Delta(T)$. Notice that $T$ is obtained from $\ell(T)$ paths $O_{1},\ldots,O_{\ell(T)}$ having $r$ as a common endvertex, in which $u_{1},\ldots,u_{\ell(T)}$ are the other endvertices, respectively. If $L(T)\cap V_{1}\neq \emptyset \neq L(T)\cap V_{2}$ for two sets $V_{1},V_{2}\in \Omega$, then $C_{2}(T)=2=(n+1)/2$. Therefore, $T\cong P_{3}$. So, we assume that $L(T)\subseteq V$ for a unique set $V\in \Omega$. We consider two cases.\vspace{0.75mm}\\
\textit{Subcase 2.1.} $C_{2}(T)\leq n-\ell(T)$. In such a situation, $C_{2}(T)=n-\ell(T)$, for otherwise $C_{2}(T)\leq(1+\Delta(T)+n-\ell(T)-1)/2<(n+1)/2$. We need to distinguish two more possibilities.\vspace{0.75mm}\\
\textit{Subcase 2.1.1.} $L(T)=V$. Since $C_{2}(T)=n-\ell(T)$, it follows that
\begin{equation}\label{par}
\Omega=\big{\{}L(T)\big{\}}\bigcup \big{\{}\{x,y\}\big{\}}\bigcup \big{\{}\{v\}\mid v\notin L(T)\cup \{x,y\}\big{\}}
\end{equation}
for some $x,y\in V(T)\setminus L(T)$. Moreover, every set in $\Omega\setminus \{L(T)\}$ forms a $2$-coalition with $L(T)$. The following claim will turn out to be useful.\vspace{0.75mm}\\
\textbf{Claim 1.} \textit{$|V(O_{i})|\leq4$ for each $i\in[\ell(T)]$.}\vspace{0.25mm}\\
\textit{Proof of Claim 1.} Suppose, to the contrary, that $|V(O_{i})|\geq5$ for some $i\in[\ell(T)]$. This in particular shows that $|V(O_{i})\setminus \{x,y,u_{i}\}|\geq2$. If $r\notin \{x,y\}$, then it is not difficult to check that there always exists a vertex in $O_{i}$ that is not $2$-dominated by $L(T)\cup \{r\}$, which is impossible. So, $r\in \{x,y\}$. Renaming the vertices if necessary, we can assume that $r=y$. Because $L(T)\cup \{x,y\}$ is a $2$-dominating set in $T$ and since $|V(O_{i})|\geq5$, it follows that $x$ belongs to $V(O_{i})$. Again, since $L(T)$ and $\{x,y\}$ form a $2$-coalition in $T$, it follows that $|V(O_{i})|=5$ (more precisely, $O_{i} \colon yaxbu_{i}$ for some vertices $a$ and $b$ in $T$). On the other hand, if there is a vertex $z\in V(O_{j})\setminus \{r,u_{j}\}$ for some $j\in \{1,\ldots,\ell(T)\}\setminus \{i\}$, then $x$ is not $2$-dominated by $L(T)\cup \{z\}$, a contradiction. In fact, we have proved that $|V(O_{i})|=5$ and $|V(O_{j})|=2$ for each $j\in \{1,\ldots,\ell(T)\}\setminus \{i\}$. With this in mind, $T$ is obtained from $P_{6}$ by joining $\ell(T)-2$ new vertices to one of its support vertices. In this case, we get $3=C_{2}(T)<(n+1)/2$, a contradiction. Thus, the statement of the claim holds. $(\square)$\vspace{0.75mm}

First, in view of Claim 1, we assume that $|V(O_{1})|=4$. Suppose that $|V(O_{i})|\geq3$ for some $i\in[\ell(T)]\setminus \{1\}$. If $r\notin \{x,y\}$, then $\{r\}$ is a $2$-coalition partner of $L(T)$. This contradicts the fact that the unique neighbor of $u_{1}$ has only one neighbor in $L(T)\cup \{r\}$. So, we must have $r\in \{x,y\}$. Let $v\in V(O_{1})\setminus \{x,y,u_{1}\}$. It is easy to see that $L(T)\cup \{v\}$ is not a $2$-dominating set in $T$ because it does not $2$-dominate any internal vertex of $O_{i}$, a contradiction. Therefore, $|V(O_{i})|=2$ for every $i\in[\ell(T)]\setminus \{1\}$. This shows that $3=C_{2}(T)=(\ell(T)+4)/2$, and hence $\ell(T)=2$. Thus, $T\cong T_{5}$.

Let $|V(O_{i})|\leq3$ for each $i\in[\ell(T)]$. If $|V(O_{1})|=\cdots=|V(O_{\ell(T)})|=2$, then $T\cong K_{1,\ell(T)}$. In such a situation, $2=C_{2}(T)=(\ell(T)+2)/2$ leads to $T\cong P_{3}$. Now assume that there exists an index $i\in[\ell(T)]$ such that $|V(O_{i})|=3$. If there is only one such index, then $T\cong S_{1,\ell(T)-1}$ with $|\Omega|=2$ due to (\ref{par}). However, this is a contradiction as $C_{2}(S_{1,\ell(T)-1})=3$. Therefore, it happens that more than one path in $\{O_{1},\ldots,O_{\ell(T)}\}$, say $O_{1},\ldots,O_{p}$, are on three vertices. If $x,y\in V(O_{i})$ for some $i\in[p]$, then $x$ or $y$ is not $2$-dominated by $L(T)\cup \{z\}$, in which $z$ is the internal vertex of a path $O_{j}$ with $j\in \{1,\ldots,p\}\setminus \{i\}$. This is a contradiction. Hence, $x$ and $y$ are internal vertices of distinct paths $O_{i}$ and $O_{j}$ for some $i,j\in[p]$. If there is an index $s\in \{1,\ldots,p\}\setminus \{i,j\}$ such that $|V(O_{s})|=3$, then the internal vertex of $V(O_{s})$ is not $2$-dominated by $L(T)\cup \{x,y\}$, a contradiction. The above argument guarantees that $T$ is obtained from $K_{1,\ell(T)}$ by subdividing two edges, each of them exactly once. So, we deduce that $3=C_{2}(T)=(\ell(T)+4)/2$, and thus $T\cong T_{5}$.\vspace{0.75mm}\\
\textit{Subcase 2.1.2.} $L(T)\varsubsetneq V$. The inequality $1+n-|V|\geq C_{2}(T)=n-\ell(T)$ implies that $V=L(T)\cup \{w\}$ for some $w\in V(T)\setminus L(T)$. This in turn shows that $\Omega=\big{\{}V\big{\}}\bigcup \big{\{}\{v\}\mid v\notin V\big{\}}$ and that each singleton set $\{v\}$ in $\Omega$ forms a $2$-coalition only with $V$. 

Assume first that $r=w$. Since $V$ is not a $2$-dominating set in $T$, there must be at least two vertices $a,b\in V(O_{i})$ for some $i\in[\ell(T)]$. If a path $O_{j}$ has an internal vertex $c$ for some $j\in \{1,\ldots,\ell(T)\}\setminus \{i\}$, then $V\cup \{c\}$ does not $2$-dominate the vertices $a$ and $b$, a contradiction. Therefore, $|V(O_{j})|=2$ for every $j\in \{1,\ldots,\ell(T)\}\setminus \{i\}$. On the other hand, in order for both $V\cup \{a\}$ and $V\cup \{b\}$ to be $2$-dominating sets in $T$, we necessarily have $|V(O_{i})|=4$. The above argument results in $3=C_{2}(T)=(\ell(T)+4)/2$, and hence $T\cong T_{5}$.

Now assume that $r\neq w$. Without loss of generality, we may assume that $w\in V(O_{1})$. Let $O_{11}$ and $O_{12}$ be the subpaths of $O_{1}$ from $r$ to $w$ and from $w$ to $u_{1}$, respectively. If $|V(O_{12})|\geq4$, then any internal vertex of $O_{12}$ is not $2$-dominated by $V\cup \{r\}$, which is impossible. If $O_{12}:wyu_{1}$ for some $y\in V(T)$, then $V(T)=V\cup \{r,y\}$ since $V\cup \{y\}$ is a $2$-dominating set in $T$. Now, $3=C_{2}(T)=(\ell(T)+4)/2$ implies that $T\cong T_{5}$. In view of this, we can assume that $|V(O_{12})|=2$. Analogously, we infer that $|V(O_{11})|\leq3$ and that $T\cong T_{5}$ if $|V(O_{11})|=3$. So, it remains for us to consider the case when $|V(O_{1})|=3$. If $|V(O_{2})|=\cdots=|V(O_{\ell(T)})|=2$, then $2=C_{2}(T)=(\ell(T)+3)/2$. So, $\ell(T)=1$, which is impossible. Therefore, there exists at least one index $i\in[\ell(T)]\setminus \{1\}$ such that $|V(O_{i})|\geq3$. A similar argument shows that $|V(O_{i})|\leq3$ for each $i\in[\ell(T)]\setminus \{1\}$. If $O_{i}:rxu_{i}$ and $O_{j}:ryu_{j}$ for some $x,y\in V(T)$ and two indices $i,j\in[\ell(T)]\setminus \{1\}$, then the vertex $x$ is not $2$-dominated by $V\cup \{y\}$, a contradiction. Therefore, $|O_{i}|=3$ for a unique index $i\in[\ell(T)]\setminus \{1\}$. In such a situation, $3=C_{2}(T)=(\ell(T)+4)/2$ implies that $T\cong T_{5}$.\vspace{0.75mm}\\
\textit{Subcase 2.2.} $C_{2}(T)=1+n-\ell(T)$. Then, $L(T)\in \Omega$ and for each $v\in V(T)\setminus L(T)$, the singleton set $\{v\}\in \Omega$ forms a $2$-coalition with $L(T)$. If $|V(Q_{1})|=\cdots=|V(Q_{\ell(T)})|=2$, then $T\cong K_{1,\ell(T)}$. Therefore, the resulting equality $2=C_{2}(T)=(\ell(T)+2)/2$ implies that $T\cong P_{3}$.

Since $\{r\}$ is a $2$-coalition partner of $L(T)$, it follows that $|V(O_{i})|\leq3$ for each $i\in[\ell(T)]$. Moreover, if two paths $O_{i}$ and $O_{j}$ with $i,j\in \{1,\ldots,\ell(T)\}$ and internal vertices $x$ and $y$, respectively, are of order~ $3$, then $x$ is not $2$-dominated by $L(T)\cup \{y\}$. This is a contradiction. In fact, by the above argument, we may assume that $|V(O_{1})|=3$ and $|V(O_{2})|=\cdots=|V(O_{\ell(T)})|=2$. Hence, $T\cong S_{1,\ell(T)-1}$. Thus, we deduce from $3=C_{2}(T)=(\ell(T)+3)/2$ that $T\cong S_{1,2}$.

All in all, we have proved that $T$ is isomorphic to one of the trees in $\{P_{2},P_{3},P_{4},P_{5},S_{1,2}\}$. This completes the proof.
\end{proof}


\section{Cubic graphs}

We observe that $C_{k}(G)=2$ if $k>3$. By the bound in Theorem~\ref{lower2}, we infer that $C_{2}(G)\ge 4$ for any cubic graph $G$. We prove that this inequality is in fact equality.

\begin{theorem}\label{thm:C2_cubic}
If $G$ is a cubic graph, then $C_2(G)=4$.
\end{theorem}
\begin{proof}
As an immediate consequence of Theorem \ref{lower2} with $k=2$, we get $C_2(G)\geq4$. So, it remains to prove that $C_{2}(G)\le 4$.

Suppose that $C_{2}(G)=t>4$ and let $\Theta=\{V_1,\ldots,V_t\}$ be a $C_{2}(G)$-partition. If all sets in $\Theta\setminus \{V_{1}\}$ form a $2$-coalition with $V_{1}$, then we have $C_{2}(G)=|\Theta|\leq4$ by Lemma \ref{lemma-k-coal} with $k=2$. So, we may assume, by renaming the sets if necessary, that $V_{2}$ does not form a $2$-coalition with $V_{1}$. Note that $V_{1}$ is a $2$-coalition partner of a set in $\Theta$, say $V_{3}$. Suppose that there exist two sets $V_{i},V_{j}\in \Theta\setminus \{V_{1},V_{3}\}$ which form a $2$-coalition. Let $v\in V_{r}$ for some $r\in[t]\setminus \{1,3,i,j\}$. Then, $3\geq|N_{G}(v)\cap(V_{1}\cup V_{3})|+|N_{G}(v)\cap(V_{i}\cup V_{j})|\geq4$, a contradiction. Therefore, each set in $\Theta\setminus \{V_{1},V_{3}\}$ forms a $2$-coalition with $V_{1}$ or $V_{3}$. In particular, $V_{2}$ is a $2$-coalition partner of $V_{3}$.

Analogously, we may assume that there is a set, say $V_{4}$, which does not form an $2$-coalition with $V_{3}$. Hence, it forms a $2$-coalition with $V_{1}$ by the above argument. In such a situation, $3\geq|N_{G}(v)\cap(V_{2}\cup V_{3})|+|N_{G}(v)\cap(V_{1}\cup V_{4})|\geq4$ for each $v\in V_{5}$, which is impossible. Thus, $|\Theta|\leq4$. This completes the proof.
\end{proof}

By the bound in Theorem~\ref{lower2}, we infer that $C_3(G)\ge 3$ for any cubic graph $G$. The exact value of $C_3(G)$ in cubic graphs depends on whether or not $G$ is bipartite.
\begin{theorem}
\label{thm:C3_cubic}
If $G$ is a cubic graph, then \[
C_{3}(G)=
\begin{cases}
4, & \textrm{if} \ G\textrm{ is bipartite};\\
3, & \textrm{otherwise}.
\end{cases}
\]
\end{theorem}
\begin{proof}
Note that $3\le C_{3}(G) \le 4$ in any cubic graph $G$. The lower bound is mentioned before the theorem, while the upper bound can be verified by an argument similar to that of the proof of Theorem \ref{thm:C2_cubic}.

First, let $G$ be a (cubic) bipartite graph with partite sets $A$ and $B$. Clearly, $|A|>2$ and $|B|>2$. Let vertices $a\in A$ and $b\in B$ be chosen arbitrarily. Let $V_1=\{a\}$, $V_2=A\setminus \{a\}$, $V_3=\{b\}$ and $V_4=B\setminus\{b\}$. It is easy to see that none of the sets $V_i$, for $1\leq i\leq4$, is a $3$-dominating set in $G$. On the other hand, $V_1$ forms a $3$-coalition with $V_2$, and $V_3$ forms a $3$-coalition with $V_4$. Therefore, $\{V_1,V_2,V_3,V_4\}$ is a $3$-coalition partition of $G$, and so $C_3(G)=4$.

Now, assume that $C_3(G)=4$, and let $\Theta=\{V_1,V_2,V_3,V_4\}$ be a $C_3(G)$-partition.  We claim that $C_3G(G,\Theta)$ contains a matching of cardinality $2$. By Lemma~\ref{lemma-k-coal} with $k=3$, the maximum degree in $C_3G(G,\Theta)$ is at most $2$. If $\Delta\big{(}C_3G(G,\Theta)\big{)}=1$, the claim trivially holds. Otherwise, assume without loss of generality that $V_1$ forms a $3$-coalition with $V_2$ and $V_3$, respectively. However, since $V_4$ forms a $3$-coalition with $V_2$ or $V_3$, a matching of cardinality $2$ in $C_3G(G,\Theta)$ is easily found.

Without loss of generality, renaming the sets if necessary, let $V_1$ form a coalition with $V_2$ and let $V_3$ form a $3$-coalition with $V_4$. Now, consider a vertex $u\in V(G)$ such that $u\in V_1\cup V_2$. Since $V_3\cup V_4$ is a $3$-dominating set in $G$, every vertex in $N_G(u)$ belongs to $V_3\cup V_4$. By an analogous argument, for any vertex $v\in V(G)$ such that $v\in V_3\cup V_4$, all of its neighbors in $G$ belong to $V_1\cup V_2$. We readily infer that $G$ is bipartite with partite sets $V_1\cup V_2$ and $V_3\cup V_4$.

Altogether, we derive that $C_3(G)\ne 4$ (or, equivalently $C_3(G)=3$ by the initial observation in this proof) if and only if $G$ is a non-bipartite graph.
\end{proof}


\section{(Total) \texorpdfstring{$k$}{k}-coalition number of complete bipartite graphs}

Unlike most of the domination-related parameters, whose values for complete bipartite graphs are trivial or easily obtained, the computation of $C_{k}$ for this family of graphs is not even straightforward. In this regard, lower and upper bounds on $C_{k}(K_{s,t})$ were given in~\cite{JAB} as follows.

\begin{theorem}{\rm \cite[Theorem 3.2]{JAB}}
\label{ajc} 
Let $K_{s,t}$ be the complete bipartite graph with $s\leq t$. Then,\vspace{0.9mm}\\
$(i)$ If $k<s$, then $\max\{s+t-4k+4,t-k+2\}\leq C_{k}(K_{s,t})\leq s+t-2k+1$.\vspace{0.75mm}\\
$(ii)$ If $1<k=s$, then $C_{k}(K_{s,t})=4$.\vspace{0.75mm}\\
$(iii)$ If $s<k\leq t$ or $t<k$, then $C_{k}(K_{s,t})=2$.
\end{theorem}

It must be noted that the statement $(i)$ in Theorem \ref{ajc} is not true as it stands. To see this, as a counterexample, it suffices to consider the complete bipartite graph $K_{3,t}$ for $k=2$ and $t\geq3$. Letting $X=\{x_{1},x_{2},x_{3}\}$ and $Y=\{y_{1},\ldots,y_{t}\}$ be the partite sets of $K_{3,t}$, we observe that 
$\Omega=\big{\{}\{x_{1},x_{2},y_{1}\},\{x_{3}\},\{y_{2}\},\ldots,\{y_{t}\}\big{\}}$ is a $2$-coalition partition of $K_{3,t}$. Therefore, $C_{2}(K_{3,t})\geq|\Omega|=t+1$, which is strictly larger than the given upper bound in Theorem \ref{ajc} for this case.

In what follows, we exhibit the exact formula for $C_{k}(K_{s,t})$ for all possible values of $s$, $t$ and $k$. 
It is readily checked that $C(K_{1,1})=2$, $C(K_{1,t})=3$ for $t\geq2$, and $C(K_{s,t})=s+t$ for $s,t\geq2$; see~\cite{hhhmm}. In view of this, we turn our attention to the case when $k\geq2$.

\begin{theorem}\label{CBG}
Let $s$, $t$ and $k$ be any positive integers with $t\geq s$ and $k\geq2$. Then,
\[
C_{k}(K_{s,t})=\begin{cases}
2, & \textrm{if}\ \ s<k;\\
4, & \textrm{if}\ \ s=k;\\
t-k+3, & \textrm{if}\ \ k+1\leq s\leq3k-2;\\
s+t-4k+4, & \textrm{if}\ \ s\geq3k-1.
\end{cases}
\]
\end{theorem}
\begin{proof}
Clearly, $C_{k}(K_{s,t})=2$ for $t<k$. So, throughout the proof, we assume that $t\geq k$. Let $X=\{x_{1},\ldots,x_{s}\}$ and $Y=\{y_{1},\ldots,y_{t}\}$ be the partite sets of $K_{s,t}$, and let $\Theta$ be a $C_{k}(K_{s,t})$-partition. Note that no set in $\Theta$ is a $k$-dominating set as $k\geq2$.

Assume first that $s<k$, and that $A,B\in \Theta$ form a $k$-coalition in $K_{s,t}$. Since $|X|=s<k$, it follows that $Y\subseteq A\cup B$. Moreover, $Y\nsubseteq A$ and $Y\nsubseteq B$. In particular, both $A\cap Y$ and $B\cap Y$ are nonempty. Notice that no two sets in $\Theta\setminus \{A,B\}$ form a $k$-coalition as $Y\subseteq A\cup B$ and $|X|<k$. So, if there is a set $C\in \Theta\setminus \{A,B\}$, we may assume that it forms a $k$-coalition with $A$. This is a contradiction as $|X|<k$ and $B\cap Y\neq \emptyset$. The above argument shows that $C_{k}(K_{s,t})=2$ when $s<k$.

Now let $s=k$. Clearly, $\big{\{}\{x_{s}\},X\setminus \{x_{s}\},\{y_{t}\},Y\setminus \{y_{t}\}\big{\}}$ is a $k$-coalition partition of $K_{s,t}$. Therefore, $C_{k}(K_{s,t})\geq4$. Since $k\geq2$, it follows that $X\notin \Theta$ and $Y\notin \Theta$. Note that there exist two sets $A,B\in \Theta$ forming a $k$-coalition in $K_{s,t}$ such that $(A\cup B)\cap X\neq \emptyset$. Suppose that $X\nsubseteq A\cup B$. Due to this, since $|X|=k$ and because $A\cup B$ is a $k$-dominating set in $K_{s,t}$, it follows that $Y\subseteq A\cup B$. With this in mind, no two sets $C,D\in \Theta\setminus \{A,B\}$ form a $k$-coalition as $A\cup B$ contains at least one vertex from $X$ and $|X|=k$. Therefore, every set in $C\in \Theta\setminus \{A,B\}$ is a $k$-coalition partner only of $A$ or $B$. Without loss of generality, we may assume that such a set $C$ forms a $k$-coalition with $A$. Since $A\cup C$ is a $k$-dominating set in $K_{s,t}$ and because $B\cap Y\neq \emptyset$, it necessarily follows that $A\cup C=X$. This implies that $|\Theta|=3$, a contradiction. Therefore, $X\subseteq A\cup B$. Moreover, both $A\cap X$ and $B\cap X$ are nonempty. If two sets $C,D\in \Theta\setminus \{A,B\}$ form a $k$-coalition in $K_{s,t}$, then we must have $Y=C\cup D$, and hence $X=A\cup B$. This implies that $|\Theta|=4$. Otherwise, every set $C\in \Theta\setminus \{A,B\}$ is a $k$-coalition partner only of $A$ or $B$. Let such a set $C$ and $A$ form a $k$-coalition. If there exists a vertex $y_{i}\in Y\setminus(A\cup C)$, then $y_{i}$ has at least $k$ neighbors in $A$. This is impossible as $|X|=k$ and $B\cap X\neq \emptyset$. Hence, $Y=A\cup C$. This contradicts the fact that $|\Theta|\geq4$. In either case, we have proved that $C_{k}(K_{s,t})=4$ when $s=k$.

We now turn our attention to the case $s\geq k+1$. It is easy to see that
\[
\big{\{}\{x_{1},\ldots,x_{k},y_{1},\ldots,y_{k-1}\},\{x_{k+1},\ldots,x_{s}\},\{y_{k}\},\ldots,\{y_{t}\}\big{\}}
\]
is a $k$-coalition partition of $K_{s,t}$ into $t-k+3$ sets. Therefore,
\begin{equation}\label{Jeeg}
C_{k}(K_{s,t})\geq t-k+3
\end{equation}
for each $s\geq k+1$. On the other hand, if $s\geq2k-1$, we set $A'=\{x_{1},\ldots,x_{k},y_{1},\ldots,y_{k-1}\}$ and $B'=\{x_{k+1},\ldots,x_{2k-1},y_{k},\ldots,y_{2k-1}\}$. It is readily checked that
\[
\{A',B'\}\cup \big{\{}\{z\}\mid z\in(X\cup Y)\setminus(A'\cup B')\big{\}}
\]
is a $k$-coalition partition of $K_{s,t}$ of cardinality $s+t-4k+4$, and hence $C_{k}(K_{s,t})\geq s+t-4k+4$ when $s\geq2k-1$. The above argument guarantees that
\begin{equation}\label{EQ22}
C_{k}(K_{s,t})\geq \max\{t-k+3,s+t-4k+4\}
\end{equation}
when $s\geq2k+1$.

The following claim will turn out to be useful in the rest of the proof. Recall that $\Theta$ is a $C_{k}(K_{s,t})$-partition. \vspace{1mm}\\
\textbf{Claim 1.} \textit{If $k+1\leq s\leq3k-2$, then there exist two sets $A,B\in \Theta$ that form a $k$-coalition in $K_{s,t}$ with the property that $X\subseteq A\cup B$ or $Y\subseteq A\cup B$.}\vspace{1mm}\\
\textit{Proof of Claim 1.} Suppose, to the contrary, that both $X\setminus(A\cup B)$ and $Y\setminus(A\cup B)$ are nonempty for all $A,B\in \Theta$ that form a $k$-coalition in $K_{s,t}$. Let $A,B\in \Theta$ be any such sets. This implies, by definition, that $|(A\cup B)\cap X|\geq k$ and $|(A\cup B)\cap Y|\geq k$. We need to distinguish two cases depending on the behavior of $s$.\vspace{1mm}\\
\textit{Case 1.} $k+1\leq s\leq2k-1$. If there exist two sets $C,D\in \Theta\setminus \{A,B\}$ that form a $k$-coalition in $K_{s,t}$, then it necessarily follows that $|(C\cup D)\cap X|\geq k$, contradicting the fact that $|(A\cup B)\cap X|\geq k$ and $s\leq2k-1$. Therefore, each set in $\Theta\setminus \{A,B\}$ forms a $k$-coalition with only $A$ or $B$.

Let $C_{1},\ldots,C_{l}$ be all the sets in $\Theta\setminus \{A,B\}$ containing at least one vertex in $X\setminus(A\cup B)$. Suppose that there exists an index $i\in[l]$ such that $C_{i}$ is a singleton set. Without loss of generality, we may assume that $C_{i}$ forms a $k$-coalition with $A$. This, in view of the fact that $X\nsubseteq A\cup C_{i}$ and $Y\nsubseteq A\cup C_{i}$, implies that $|A\cap X|\geq k-1$ and $|A\cap Y|\geq k$. In addition, we have the equality $|A\cap X|=k-1$ as $A$ is not a $k$-dominating set in $K_{r,s}$.

Let $D$ be any set in $\Theta\setminus \{A,B\}$ that contains a vertex from $Y\setminus(A\cup B)$. If the set $D$ forms a $k$-coalition with $A$, then it must contain at least one vertex from $X\setminus(A\cup B)$ because of the fact that $Y\nsubseteq A\cup D$ and $|A\cap X|=k-1$. If $D$ forms a $k$-coalition with $B$, we deduce again that it contains at least one vertex from $X\setminus(A\cup B)$ as $Y\nsubseteq A\cup D$ and $|B\cap X|\leq k-1$. Letting $D_{1},\ldots,D_{p}\in \Theta$ be all such sets, we infer that $|\Theta|\leq2+p+s-k-p<3+t-k$, in contradiction to (\ref{Jeeg}). Therefore, $|C_{i}|\geq2$ for each $i\in[l]$. Analogously, every set in $\Theta\setminus \{A,B\}$ that contains a vertex from $Y\setminus(A\cup B)$ has at least two vertices. Summing up, we get
\[
|\Theta|\leq2+(s+t-|A\cup B|)/2\leq2+(s+t-2k)/2=2+(s+t)/2-k<3+t-k,
\]
a contradiction.\vspace{1mm}\\
\textit{Case 2.} $2k\leq s\leq3k-2$. Let $C$ be any set in $\Theta\setminus \{A,B\}$. If $C$ forms a $k$-coalition with a set $D\in \Theta\setminus \{A,B\}$, then both $(C\cup D)\cap X$ and $(C\cup D)\cap Y$ have at least $k$ vertices. In such a situation, we get $|\Theta|\leq4+s+t-4k<t-k+3$ since $s\leq3k-2$, contradicting (\ref{Jeeg}). So, we may assume that every set $C\in \Theta\setminus \{A,B\}$ is a $k$-coalition partner only of $A$ or $B$. Next, we consider the sets $C_{1},\ldots,C_{l}$ and $D_1,\ldots,D_p$ as defined in Case 1. We distinguish two cases.\vspace{1mm}\\
\textit{Subcase 2.1.} There exists a singleton set $C_{i}$ for some $i\in[l]$. We may assume, without loss of generality, that $C_{i}$ is a $k$-coalition partner of $A$. This implies that $|A\cap X|=k-1$ and $|A\cap Y|\geq k$. If all the sets $D_{1},\ldots,D_{p}$ form a $k$-coalition with $A$, each of them has at least one vertex from $X\setminus(A\cup B)$ because $|A\cap X|=k-1$ and $Y\cup D_{j}\neq \emptyset$ for each $j\in[p]$. In such a situation, we have $|\Theta|\leq2+p+s-k-p<3+t-k$, which contradicts the inequality (\ref{Jeeg}). In view of this, we turn our attention to the existence of some sets among $D_{1},\ldots,D_{p}$ that form $k$-coalitions with $B$. Renaming indices if necessary, we may assume that $D_{a+1},\ldots,D_{p}$ are $k$-coalition partners of $B$ for some $a\in \{0,\ldots,p-1\}$.

If there is an index $j\in \{a+1,\ldots,p\}$ such that $D_{j}$ is a singleton set, then we necessarily have $|B\cap X|\geq k$ and $|B\cap Y|=k-1$. Therefore, $|A\cup B|\geq4k-2$. Invoking this fact, we get $|\Theta|\leq2+s+t-4k+2<t-k+3$ as $s\leq3k-2$, which is impossible. Hence, $|D_{j}|\geq2$ for every $j\in \{a+1,\ldots,p\}$. We note that,\vspace{-2mm}
\begin{itemize}
\item[$(i)$] every $D_{i}$, with $i\leq a$, contains a vertex in $X\setminus(A\cup B)$ (as proved in the first paragraph of Subcase 2.1), and\vspace{-2mm}
\item[$(ii)$] the fact that $D_{j}$, with $j\in \{a+1,\ldots,p\}$, forms a $k$-coalition with $B$ shows that both $|(D_{j}\cup B)\cap X|$ and $|(D_{j}\cup B)\cap Y|$ are at least $k$.
\end{itemize}\vspace{-2mm}
Keeping $(i)$ and $(ii)$ in mind, it is now readily seen that
\[
A \cup B\cup D_{j} \cup \big{(} \bigcup_{i\in[p]\setminus \{j\}}D_{i} \big{)}
\]
contains at least $4k-1+2(p-1)$ vertices. Thus, $|\Theta|\leq2+p+s+t-(4k-1)-2(p-1)\leq s+t-4k+4<t-k+3$ in view of the fact that $p\geq1$ and $s\leq3k-2$, a contradiction.\vspace{1mm}\\
\textit{Subcase 2.2.} $|C_{i}|\geq2$ for every $i\in[l]$. Similarly to Subcase 2.1 (with $D_{i}$ instead of $C_{i}$), we reach a contradiction if there is a singleton set $D_{i}$ for some $i\in[p]$. Therefore, $|D_{i}|\geq2$ for every $i\in[p]$. This implies that $l+p\leq(s+t-|A\cup B|)/2$, and thus
\[
|\Theta|=2+l+p\leq2+(s+t-|A\cup B|)/2\leq2+(2t-2k)/2=t-k+2<t-k+3,
\]
a contradiction. $(\square)$\vspace{1mm}

Assume now, in view of Claim 1, that $A,B\in \Theta$ form a $k$-coalition in $K_{s,t}$ such that $X\subseteq A\cup B$ or $Y\subseteq A\cup B$. Without loss of generality, we can assume that $X\subseteq A\cup B$. Let there be two sets $C,D\in \Theta\setminus \{A,B\}$ that form a $k$-coalition. Since $X\subseteq A\cup B$ and because $C\cup D$ is a $k$-dominating set in $K_{s,t}$, it follows that $X=A\cup B$ and $Y=C\cup D$. Therefore, $|\Theta|=4\leq t-k+3$. So, we assume that every set in $\Theta\setminus \{A,B\}$ is a $k$-coalition partner only of $A$ or $B$. Let $C\in \Theta\setminus \{A,B\}$ form a $k$-coalition with $A$. Since $X\setminus A\neq \emptyset$, it follows that $|(A\cup C)\cap Y|\geq k$. Therefore, $|\Theta|\leq 3+t-k$. In each case, we have the desired equality due to (\ref{Jeeg}).

It remains for us to consider the case $s\geq3k-1$. This together with the inequality (\ref{EQ22}) leads to $C_{k}(K_{s,t})\geq s+t-4k+4$. If there exist two sets $C,D\in \Theta\setminus \{A,B\}$ that form a $k$-coalition in $K_{s,t}$, we get $|\Theta|\leq4+s+t-4k$ as it is easy to see that $|A\cup B|+|C\cup D|\geq4k$. Therefore, we can assume that each set in $\Theta\setminus \{A,B\}$ forms a $k$-coalition only with $A$ or $B$. In such a situation, an argument similar to that of Case 2 in the proof of Claim 1 reaches a contradiction or leads to the desired inequality $|\Theta|\leq s+t-4k+4$. Therefore, we get the desired equality $C_{k}(K_{s,t})=|\Theta|=4+s+t-4k$ when $s\geq3k-1$. This completes the proof.
\end{proof}


A \textit{total $k$-dominating set} $S$ and a \textit{total $k$-coalition} in a graph $G$ with $\delta(G)\geq k$ can be defined by replacing the condition ``every vertex $v\in V(G)\setminus S$", in the definition of a $k$-dominating set, with ``every vertex $v\in V(G)$". A \textit{total $k$-coalition partition} in $G$ is a partition $\Omega$ of $V(G)$ in which every set forms a total $k$-coalition with another set. The \textit{total $k$-coalition number} $TC_{k}(G)$ equals the maximum cardinality taken over all total $k$-coalition partitions in $G$. When $k=1$, this concept coincides with total coalition in graphs (see \cite{HJ} and the references therein). This concept, as a sequel to total coalition and as the total version of $k$-coalition, was investigated in \cite{bsb}.

Using the techniques from the proof of Theorem~\ref{CBG}, we can prove a similar result for the total $k$-coalition number of complete bipartite graphs. However, for the sake of completeness, we present the proof in detail.

\begin{theorem}\label{Total}
For any integers $s$, $t$ and $k$ with $t\geq s\geq k$,
\[
TC_{k}(K_{s,t})=\begin{cases}
t-k+2, & \textrm{if}\ \ k\leq s\leq3k-2;\\
s+t-4k+4, & \textrm{if}\ \ s\geq3k-1.
\end{cases}
\]
\end{theorem}
\begin{proof}
Let $X=\{x_{1},\ldots,x_{s}\}$ and $Y=\{y_{1},\ldots,y_{t}\}$ be the partite sets of $K_{s,t}$. We first observe, by the structure of bipartite graphs and the definition of total $k$-coalition, that $|(A\cup B)\cap X|\ge k$ and $|(A\cup B)\cap Y|\geq k$ for any two subsets $A,B\subseteq V(K_{s,t})$ that form a total $k$-coalition in $K_{s,t}$. It is easy to see that
\[
\big{\{}X\cup \{y_{1},\ldots,y_{k-1}\},\{y_{k}\},\ldots,\{y_{t}\}\big{\}}
\]
is a total $k$-coalition partition of $K_{s,t}$ into $t-k+2$ sets. Therefore,
\begin{equation}\label{Jeeg2}
\TC_{k}(K_{s,t})\geq t-k+2
\end{equation}
for each $s\geq k$. If $s\geq2k-1$, then we set 
\[
A'=\{x_{1},\ldots,x_{k},y_{1},\ldots,y_{k-1}\} \hspace*{0.25cm} \mbox{and}  \hspace*{0.25cm} B'=\{x_{k+1},\ldots,x_{2k-1},y_{k},\ldots,y_{2k-1}\}.
\]
It is then readily checked that
\[
\{A',B'\}\cup \big{\{}\{z\}\mid z\in(X\cup Y)\setminus(A'\cup B')\big{\}} 
\]
is a total $k$-coalition partition of $K_{s,t}$ of cardinality $s+t-4k+4$, and hence $\TC_{k}(K_{s,t})\geq s+t-4k+4$ when $s\geq2k-1$. The above argument guarantees that
\begin{equation}\label{EQ11}
\TC_{k}(K_{s,t})\geq \max\{t-k+2,s+t-4k+4\}
\end{equation}
for $t\geq s\geq2k-1$.

Let $\Theta$ be a $\TC_{k}(K_{s,t})$-partition and let $A,B$ be two sets of $\Theta$. We distinguish three cases depending on $s$.\vspace{1mm}\\
\textit{Case 1.} $s\leq2k-1$. If there exist two sets $C,D\in \Theta\setminus \{A,B\}$ that form a total $k$-coalition in $K_{s,t}$, then it necessarily follows that $|(C\cup D)\cap X|\geq k$, contradicting the fact that $|(A\cup B)\cap X|\geq k$ and $s\leq2k-1$. Therefore, each set in $\Theta\setminus \{A,B\}$ forms a total $k$-coalition only with $A$ or $B$.

If $X\subseteq A\cup B$, then $|V(K_{s,t})|-|A\cup B|\leq t-k$. In such a situation, it is clear that $|\Theta|\leq2+t-k$. Therefore, $\TC_{k}(K_{s,t})=t-k+2$ by (\ref{Jeeg2}). Furthermore, in a similar fashion, we have the desired equality if $Y\subseteq A\cup B$. Due to this, we may assume that $X\setminus(A\cup B)\ne\emptyset$ and $Y\setminus(A\cup B)\neq \emptyset$.

Let $C_{1},\ldots,C_{l}$ be all the sets in $\Theta\setminus \{A,B\}$ that contain at least one vertex in $X\setminus(A\cup B)$. Assume that there exists an index $i\in[l]$ such that $C_{i}$ is a singleton set. Without loss of generality, we may assume that $C_{i}$ forms a total $k$-coalition with $A$. This necessarily implies that $|A\cap X|=k-1$ and $|A\cap Y|\geq k$. Let $D$ be any set in $\Theta\setminus \{A,B\}$ that contains a vertex from $Y\setminus(A\cup B)$. If it forms a total $k$-coalition with $A$, then the set $D$ must contain at least one vertex from $X\setminus(A\cup B)$ as $|A\cap X|=k-1$. If it forms a total $k$-coalition with $B$, we again deduce that it contains at least one vertex from $X\setminus(A\cup B)$ as $|B\cap X|\leq k-1$. Letting $D_{1},\ldots,D_{p}\in \Theta$ be such sets, we infer that $|\Theta|\leq2+p+s-k-p\leq2+t-k$. In view of this, we may assume that $|C_{i}|\geq2$ for each $i\in[l]$. Analogously, we can assume that every set in $\Theta\setminus \{A,B\}$ that contains a vertex from $Y\setminus(A\cup B)$ has at least two vertices. This implies that
\begin{center}
$|\Theta|=2+\ell+p\leq2+(s+t-2k)/2=2+(s+t)/2-k\leq2+t-k$,
\end{center}
resulting in $\TC_{k}(K_{s,t})=t-k+2$.\vspace{1mm}\\
\textit{Case 2.} $2k\leq s\leq3k-2$. This trivially implies that $k\geq2$. Assume first that there exist two sets $C,D\in \Theta\setminus \{A,B\}$ that form a total $k$-coalition in $K_{s,t}$. Since both $|A\cup B|$ and $|C\cup D|$ are greater than or equal to $2k$, it follows that $\TC_{k}(K_{s,t})=|\Theta|\leq4+s+t-4k\leq t-k+2$. We now assume that every set in $\Theta\setminus \{A,B\}$ forms a total $k$-coalition only with $A$ or $B$.

As before, let $C_{1},\ldots,C_{l}$ and $D_{1},\ldots,D_{p}$ be all the sets in $\Theta\setminus \{A,B\}$ that have at least one vertex from $X\setminus(A\cup B)$ and $Y\setminus(A\cup B)$, respectively. We need to consider two more possibilities depending on the behavior of $C_{1},\ldots,C_{l}$.\vspace{1mm}\\
\textit{Subcase 2.1.} There exists a singleton set $C_{i}$ for some $i\in[l]$. We may assume, without loss of generality, that $C_{i}$ is a total $k$-coalition partner of $A$. This implies that $|A\cap X|=k-1$ and $|A\cap Y|\geq k$. If all the sets $D_{1},\ldots,D_{p}$ form a total $k$-coalition with $A$, each of them has at least one vertex from $X\setminus(A\cup B)$ because $|A\cap X|=k-1$. In such a situation, $|\Theta|\leq2+p+s-k-p\leq2+t-k$, resulting in the desired equality due to (\ref{Jeeg2}). In view of this, we turn our attention to the existence of some sets among $D_{1},\ldots,D_{p}$ that form total $k$-coalitions with $B$. Renaming indices if necessary, we may assume that $D_{a+1},\ldots,D_{p}$ are total $k$-coalition partners of $B$ for some $a\in \{0,\ldots,p-1\}$.

If there is an index $j\in \{a+1,\ldots,p\}$ such that $D_{j}$ is a singleton set, then we necessarily have $|B\cap X|\geq k$ and $|B\cap Y|=k-1$. Therefore, $|A\cup B|\geq4k-2$. Invoking this fact, we get $|\Theta|\leq2+s+t-4k+2\leq t-k+2$ as $s\leq3k-2$. Hence, we can assume that $|D_{j}|\geq2$ for every $j\in \{a+1,\ldots,p\}$. By the above arguments, we have\vspace{-2mm}
\begin{itemize}
\item[$(i)$] every $D_{i}$, with $i\leq a$, contains a vertex in $X\setminus(A\cup B)$ (as proved in the first paragraph of Subcase 2.1), and\vspace{-2mm}
\item[$(ii)$] the fact that $D_{j}$, with $j\in \{a+1,\ldots,p\}$, forms a $k$-coalition with $B$ shows that both $|(D_{j}\cup B)\cap X|$ and $|(D_{j}\cup B)\cap Y|$ are at least $k$.
\end{itemize}\vspace{-2mm}
With $(i)$ and $(ii)$ in mind, it is now easily observed that
\[
A \cup B \cup D_{j} \cup \big{(}\bigcup_{i\in[p]\setminus \{j\}}D_{i}\big{)}
\]
contains at least $4k-1+2(p-1)$ vertices. Thus, $|\Theta|\leq2+p+s+t-(4k-1)-2(p-1)\leq t-k+2$ in view of the fact that $p\geq1$ and $s\leq3k-2$.\vspace{1mm}\\
\textit{Subcase 2.2.} $|C_{i}|\geq2$ for every $i\in[l]$. Similarly to Subcase 2.1, we derive the upper bound $t-k+2$ for $|\Theta|$ if there is a singleton set $D_{i}$ for some $i\in[p]$. In view of this, we may also assume that $|D_{i}|\geq2$ for every $i\in[p]$. In such a situation, we have $l+p\leq(s+t-|A\cup B|)/2$, and thus
\begin{center}
$|\Theta|=2+l+p\leq2+(s+t-|A\cup B|)/2\leq2+(2t-2k)/2=t-k+2$.
\end{center}

In either case, we have shown that $\TC_{k}(K_{s,t})=t-k+2$ when $k\leq s\leq3k-2$.\vspace{1mm}\\
\textit{Case 3.} $s\geq3k-1$. In this case, we infer from (\ref{EQ11}) that $\TC_{k}(K_{s,t})\geq s+t-4k+4$. If there exist two sets $C,D\in \Theta\setminus \{A,B\}$ that form a total $k$-coalition in $K_{s,t}$, we get $|\Theta|\leq4+s+t-4k$ because both $A\cup B$ and $C\cup D$ have at least $2k$ vertices. Therefore, we can assume that each set in $\Theta\setminus \{A,B\}$ forms a total $k$-coalition only with $A$ or $B$. In this situation, an argument similar to that in Case 2 shows that $|\Theta|\leq4+s+t-4k$ or $|\Theta|\leq t-k+2$. However, $t-k+2\leq4+s+t-4k$ as $s\geq3k-1$. Therefore, we get the desired equality $\TC_{k}(K_{s,t})=|\Theta|=4+s+t-4k$. This completes the proof.
\end{proof}



\begin{thebibliography}{99}

\bibitem{bhp} D.~Bakhshesh, M.A.~Henning and D.~Pradhan, {\em On the coalition number of trees}, Bull. Malays. Math. Sci. Soc. {\bf 46} (2023), Paper No. 95.

\bibitem{bsb} B.~Bre\v{s}ar, S.~Klav\v{z}ar and B.~Samadi, {\em Total $k$-coalition: bounds, exact values and an application to double coalition}, Discrete Math. Theor. Comput. Sci. \textbf{27} (2025), Paper No. 3.



\bibitem{HaVo-20} A. Hansberg and L. Volkmann, Multiple domination. \emph{Topics in domination in graphs}, 151--203. \emph{Dev. Math.}, \textbf{64}. Springer, Cham, [2020] \copyright 2020.

\bibitem{hhhmm} T. W.~Haynes, J. T.~Hedetniemi, S. T.~Hedetniemi, A. A.~McRae and R.~Mohan, {\em Introduction to coalitions in graphs}, AKCE Int. J. Graphs Combin. {\bf 17} (2020), 653--659.

\bibitem{HaHeHe-20} T. W. Haynes, S. T. Hedetniemi, and M. A. Henning  (eds), \emph{Topics in Domination in Graphs}. Series: Developments in Mathematics, Vol. 64, Springer, Cham, 2020. viii + 545 pp.

\bibitem{HaHeHe-21} T. W. Haynes, S. T. Hedetniemi, and M. A. Henning  (eds), \emph{Structures of Domination in Graphs}. Series: Developments in Mathematics, Vol. 66, Springer, Cham, 2021. viii + 536 pp.

\bibitem{HaHeHe-23} T. W. Haynes, S. T. Hedetniemi, and M. A. Henning, \emph{Domination in Graphs: Core Concepts} Series: Springer Monographs in Mathematics, Springer, Cham, 2023. xx + 644 pp.

\bibitem{HJ} M. A.~Henning and S. N.~Jogan, {\em A characterization of graphs with given total coalition numbers}, Discret. Appl. Math. {\bf 358} (2024), 395--403.


\bibitem{JAB} A.~Jafari, S.~Alikhani and D.~Bakhshesh, {\em $k$-coalitions in graphs}, Australas. J. Combin. \textbf{92} (2025), 194--209. 



\bibitem{West} D. B.~West, Introduction to Graph Theory (Second Edition), Prentice Hall, USA, 2001.

\bibitem{YSZ} F.~Yang, Q.~Sun and C.~Zhang, {\em Characterizations of graphs with equal vertex and coalition number}, Available at SSRN: https://ssrn.com/abstract=5277305 or http://dx.doi.org/10.2139/ssrn.5277305


\end{thebibliography}
\end{document}